\documentclass[10pt]{amsart}
\usepackage{kantlipsum} 
\usepackage{url}
\calclayout
\usepackage{mathrsfs}
\usepackage{amssymb}
\usepackage{enumitem}
 \usepackage{amsmath}
\usepackage{physics}
\DeclareFontFamily{U}{wncy}{}
\DeclareFontShape{U}{wncy}{m}{n}{<->wncyr10}{}
\DeclareSymbolFont{mcy}{U}{wncy}{m}{n}
\DeclareMathSymbol{\Sh}{\mathord}{mcy}{"58} 

\usepackage{color}

\newtheorem*{claim}{Claim}
\theoremstyle{plain}
\newtheorem{thm}{Theorem}[section]

\newtheorem{lem}[thm]{Lemma}
\newtheorem{cor}[thm]{Corollary}

\newtheorem{prop}[thm]{Proposition}

\theoremstyle{definition}
\newtheorem{defn}[thm]{Definition}

\newtheorem*{remark}{Remark}

\theoremstyle{remark}
\newtheorem{rem}[thm]{Remark}

	\theoremstyle{thm}

\newtheorem{Theorem}{Theorem}

\newtheorem{Corollary}[Theorem]{Corollary}

\newcommand{\cf}{\textit{cf. }}
\newcommand{\ie}{\textit{i.e. }}

\newcommand{\eg}{\textit{eg. }}
\DeclareMathOperator{\PGL}{PGL}
\DeclareMathOperator{\Q}{\mathbf{Q}}
\DeclareMathOperator{\Z}{\mathbf{Z}}

\DeclareMathOperator{\R}{\mathbf{R}}
\DeclareMathOperator{\C}{\mathbf{C}}

\DeclareMathOperator{\cusps}{cusps}
\DeclareMathOperator{\Cl}{Cl}
\DeclareMathOperator{\supp}{Supp}

\DeclareMathOperator{\SL}{SL}
\DeclareMathOperator{\Gal}{Gal}

\DeclareMathOperator{\Sel}{Sel}

\def\Dd{\mathcal{D}}

\def\q{{\mathfrak q}}

\def\l{{\mathfrak l}}
\def\Ll{{\mathfrak L}}

\author{Emmanuel Lecouturier}
 \address{Westlake University, Institute for Theoretical Sciences, Yungu Campus, 600 Dunyu Rd., Xihu District, Hangzhou, 310030, Zhejiang province, PR China}
\email{elecoutu@westlake.edu.cn}

\author{Christian Maire}
 \address{Université Marie et Louis Pasteur,  CNRS, Institut FEMTO-ST, F-25000 Besançon, France}
\email{christian.maire@univ-fcomte.fr}

\title{Indivisibility of Ray Class groups of real quadratic fields}
\begin{document}

\date{\today}

\begin{abstract}
Let $\ell,p \geq 5$ be primes such that $p \mid (\ell- 1)$. Let $\Delta>0$ be the fundamental discriminant of a real quadratic field in which $\ell$ splits. We denote by $h_\ell^-(\Delta)$ the order of the  minus part (for the Galois action) of the ray class group of $\Q(\sqrt{\Delta})$ of modulus $\ell$. In this paper, we study the indivisibility of $h_\ell^-(\Delta) $ by $p$, and prove that
$$\displaystyle{\#\{0<\Delta<X, \  \left(\dfrac{\Delta}{\ell}\right)=1 \text{ and } p\nmid h^-_\ell(\Delta)\}\gg \frac{\sqrt{X}}{\log(X)}} ,$$
under the assumption that this set is non-empty. This lower bound is made unconditional if $\ell=2p+1$, \ie if $p$ is a Sophie Germain prime. Our result can be viewed as being in the continuity of the results of Kohnen--Ono, Ono, Byeon, Beckwith etc. regarding the class numbers of quadratic fields, in the sense that we rely on techniques from the theory of half-integral weight modular forms. Significant difficulties however arise in our study, as we have to study Eisenstein congruences for cuspforms of weight $\frac{3}{2}$, and use a generalized Shimura correspondence of Baruch--Mao. Combined with the results of Lecouturier--Wang, our result has implications \eg for the $5$-part of BSD for even quadratic twists of $X_0(11)$.
\end{abstract}


\maketitle


\section*{Introduction}
The class group $\Cl_K$ of a number field $K$ is a finite abelian group that plays a central role in many questions in number theory, however its structure largely remains mysterious, except in a few very specific situations.
To improve its understanding, Cohen and Lenstra \cite{Cohen-Lenstra} proposed heuristics, particularly for families of quadratic fields.
Let $h(\Delta)$  denote the class number of the quadratic field $\Q(\sqrt{\Delta})$ with fundamental discriminant~$\Delta$.
Let $p>2$  be a prime number. Set 
$$\displaystyle{\eta^k(p):=\prod_{i\geq k} \left(1-\frac{1}{p^i}\right)} .$$
The Cohen–Lenstra  heuristics predict that the proportion of imaginary quadratic fields (respectively real quadratic fields) such that $p\nmid h(\Delta)$ is equal to $\eta^1(p)$ (resp. $\eta^2(p)$). A notable result regarding the indivisibility by $p$ of the class number of quadratic fields, though still far from approaching these heuristics, was obtained by Davenport and Heilbronn \cite{Davenport-Heilbronn} and $p=3$, as well as Hartung \cite{Hartung} for imaginary quadratic fields and $p>2$. These results have been extended to any prime $p$ by Horie \cite{Horie}, under the condition that $p$ does not split in the family of quadratic fields considered.

A breakthrough came with a paper of Kohnen and Ono \cite{Kohnen-Ono} in the late 1990s, where  they showed that for a prime $p>2$ and $X$ sufficiently large, 
$$\displaystyle{\#\{-X<\Delta<0 \text{ such that } p\nmid h(\Delta)\} \gg \frac{\sqrt{X}}{\log X}}.$$
The proof relies on the fact that, in this setting, the class number formula relates the order $h(\Delta)$ of the class group of
 $K=\Q(\sqrt{\Delta})$ to the special value $\zeta_K(s)$ of the Dedekind zeta function of $K$ at $s=0$. They then use a well-known result dating back from Gauss to relate $h(\Delta)$ to some Fourier coefficient of a modular form of weight~$3/2$. 
 
 Following the same principle, for real quadratic fields, Ono \cite{Ono} then Byeon \cite{Byeon} also obtained a result of indivisibility by making use of modular forms of weight $3/2$, and by using, in particular, Sturm's bound.
 More recently, Wiles \cite{Wiles} revisited the question of indivisibility by $p$ of the class number of imaginary quadratic fields satisfying a given set of local conditions. This result was further refined by Beckwith \cite{Beckwith}, who specified the number of such imaginary quadratic fields, again using modular forms.

\medskip

We now consider a situation in the spirit of these works.

Let $\ell,p \geq 5$ be  primes such that $p \mid (\ell-1)$. Let $\Delta>0$ be a fundamental discriminant such that  $\left(\frac{\Delta}{\ell}\right)=1$, equivalently $\ell$ splits in $K:=\Q(\sqrt{\Delta})$.
Let $\Cl_{K,\ell}$ denotes the ray class group of $K$ with modulus $\ell$.
The Galois group $\Gal(K/\Q)$ acts on $\Cl_{K,\ell}$, and we denote by $h^-_{\ell}(\Delta)$ the order of the eigenspace for the eigenvalue $-1$ of the non-trivial element of $\Gal(K/\Q)$. We prove the following.

\medskip

\begin{Theorem}\label{main_thm_paper}
Let  $\ell$, $p$ be primes $\geq 5$ such that $p \mid (\ell-1)$. Let $Q$ be a squarefree positive integer such that $\gcd(Q, 2p\ell)=1$ and $Q \not\equiv \pm 1 \text{ (modulo }p\text{)}$. Consider the set
$$\mathcal{D}(\ell, p, Q) := \{\Delta>0 \text{ fundamental discriminant s.t. } \left(\dfrac{\Delta}{Q}\right)=-1, \left(\dfrac{\Delta}{\ell}\right)=1 \text{ and }p\nmid h^-_\ell(\Delta)\} \text{.}$$
Assume that $\mathcal{D}(\ell, p, Q) \neq \emptyset$. Then we have
$$\displaystyle{\#\{0<\Delta<X \text{ with } \Delta \in \mathcal{D}(\ell, p, Q)\}\gg_{p,\ell, Q} \frac{\sqrt{X}}{\log(X)}} \text{.}$$
\end{Theorem}

\medskip

A similar result without the extra parameter $Q$ holds as an immediate corollary, as if $p\nmid h^-_\ell(\Delta)$ and $\left(\dfrac{\Delta}{\ell}\right)=1$, then one can always find a $Q$ satisfying the conditions of the theorem such that $\Delta \in \mathcal{D}(\ell, p, Q)$. For convenience, we record it in the following

\medskip

\begin{Corollary}\label{main_thm_cor}
Let  $p$, $\ell$ be primes $\geq 5$ such that $p \mid (\ell-1)$. Consider the set
$$\mathcal{D}(\ell,p) := \{\Delta>0 \text{ fundamental discriminant s.t. } \left(\dfrac{\Delta}{\ell}\right)=1 \text{ and }p\nmid h^-_\ell(\Delta)\} \text{.}$$
Assume that $\mathcal{D}(\ell,p) \neq \emptyset$. Then we have
$$\displaystyle{\#\{0<\Delta<X \text{ with } \Delta \in \mathcal{D}(\ell,p)\}\gg_{p,\ell} \frac{\sqrt{X}}{\log(X)}} \text{.}$$
\end{Corollary}

\medskip

\begin{remark}
It is not hard to see (\cf Proposition \ref{prop_link_h_ell_unit}) that the condition $\Delta \in \mathcal{D}(\ell,p)$ is equivalent to $p \nmid h(\Delta)$ and any fundamental unit of $\Q(\sqrt{\Delta})$ is not a $p$th power modulo a prime above $\ell$.
\end{remark}
We also get the following corollary in the case of Sophie Germain primes (choosing $Q=3$ in our theorem).

\medskip

\begin{Corollary}\label{main_cor}
Let $p \geq 5$ be a Sophie Germain prime, \ie such that $\ell := 2p+1$ is also prime. Then
$$\displaystyle{\#\{0<\Delta<X  \text{ fundamental discriminant s.t. } \left(\dfrac{\Delta}{3}\right)=-1 \text{,} \left(\dfrac{\Delta}{\ell}\right)=1 \text{ and }p\nmid h^-_\ell(\Delta)\}\gg_{p} \frac{\sqrt{X}}{\log(X)} } \text{.}$$
\end{Corollary}
This corollary is obtained by constructing explicitly some $\Delta_0$ in $\mathcal{D}(p, \ell, 3)$ (\cf Theorem \ref{thm_SophieGermain} below). 

\medskip

As Cohen and Lenstra explain in \cite[\S 8]{Cohen-Lenstra}, while the heuristic for imaginary quadratic fields is quite natural, the one for real quadratic fields is more subtle. An analogous situation given at the end of \S 8 of \cite{Cohen-Lenstra}, provided by Gross, supports this principle. We believe our context can also be seen as an example that supports this view. We expect that the equality
$$\lim_{X\rightarrow +\infty}\frac{\displaystyle{\#\{0<\Delta<X \text{ with } \Delta \in \mathcal{D}(\ell,p)\}}}{\displaystyle{\#\{0<\Delta<X \text{ with }\left(\dfrac{\Delta}{\ell}\right)=1 \}}} \stackrel{?}{=} \eta^1(p) $$
holds. In other words, among the fundamental discriminants 
$\Delta>0$ satisfying $\left(\frac{\Delta}{\ell}\right)=1$, the density of those for which $p\nmid h^-_\ell(\Delta)$ holds should be equal to  $\eta^1(p)$. In particular, this conjectural density is independent of $\ell$. We give some numerical data supporting this conjecture in \S \ref{section_heuristic}. Recently, Bartel and Pagano \cite{Bartel-Pagano} addressed the Cohen--Lenstra conjectures for the ray class groups of quadratic fields for a fixed rational modulus. It would be interesting to revisit their results within the framework of our study.

\medskip

Finally, using the results of \cite{Lecouturier_Wang}, we can deduce from a slight modification of Theorem \ref{main_thm_paper} some consequences for the $p$-part of the BSD conjecture for the $p$-Eisenstein quotient of $J_0(\ell)$, assuming that the pair $(\ell, p)$ (with $N=\ell$ in the notation of \cite{Lecouturier_Wang}) satisfies \cite[Hypothesis H]{Lecouturier_Wang}. This Hypothesis H is satisfied in particular for $\ell=11$ and $p=5$, and in this case $J_0(\ell)$ is the elliptic curve $E: y^2+y=x^3-x^2-10x-20$ over $\Q$. If $\Delta$ is the discriminant of a quadratic field, we denote by $E^{(\Delta)}$ the corresponding quadratic twist over $\Q$. The result we get is the following.

\medskip

\begin{Corollary}\label{cor_BSD}
Consider the elliptic curve $E = X_0(11): y^2+y=x^3-x^2-10x-20$ over $\Q$. Then
\begin{align*}
  \#\{\,
  & 0<\Delta<X  \text{ such that } \left(\dfrac{\Delta}{5}\right)=-1\text{, } \left(\dfrac{\Delta}{11}\right)=1 \text{, } \\
  & \Sh(E^{(\Delta)}/\Q)[5]=0 \text{, } E^{(\Delta)}(\Q) \otimes_{\Z} \Z/5\Z=0 \text{ and the }5\text{-part of BSD holds for }E^{(\Delta)}/\Q \} \gg \frac{\sqrt{X}}{\log(X)} \text{.}
\end{align*}
\end{Corollary} 

\

In line with the Cohen--Lenstra type heuristic above for the divisibility of $h_{\ell}^-(\Delta)$, and the fact that $100\%$ of $0<\Delta<X$ with $\left(\dfrac{\Delta}{11}\right)=1$ are expected to satisfy $E^{(\Delta)}(\Q) \otimes_{\Z} \Z/5\Z=0$, we are led to expect that 
$$\lim_{X\rightarrow +\infty}\frac{\displaystyle{\#\{0<\Delta<X \text{ with } \left(\dfrac{\Delta}{5}\right)=-1  \text{, } \left(\dfrac{\Delta}{11}\right)=1 \text{ and } \\
  \Sh(E^{(\Delta)}/\Q)[5]=0 \}}}{\displaystyle{\#\{0<\Delta<X \text{ with }\left(\dfrac{\Delta}{5}\right)=-1 \text{ and }\left(\dfrac{\Delta}{11}\right)=1  \}}} \stackrel{?}{=} \eta^1(5)\approx 0.7680 \text{.}$$

\

\

Let us outline the general principle of the proof of the Theorem \ref{main_thm_paper}. There are three main steps. \\

The first step is to relate the indivisibility of $h_\ell^-(\Delta)$ by $p$ to the indivisibility of suitably normalized central critical $L$-value $L(F, \chi_{\Delta}, 1)$, where $F \in S_2(\Gamma_0(\ell))$ is a Hecke newform which is congruent to an Eisenstein series modulo $p$. Such a newform is known to exist by the Eisenstein ideal theory due to Mazur \cite{Mazur_Eisenstein}, but it may not be unique (up to conjugation) unless $p^2 \nmid (\ell-1)$. The congruence for the $L$-value has been proved in an earlier work of the first author (\cf \cite[(26)]{Lecouturier_HV}), although in a form which is not quite suitable for our present purpose, as it makes use of a period and requires $p^2 \nmid (\ell-1)$. In this paper, follow the method of \cite{Lecouturier_HV}, namely we use an explicit form of a formula of Waldspurger, due to Popa \cite{Popa_L}, for $L(F/\Q(\sqrt{\Delta}),1)$ as well as congruences for modular symbols (essentially due to Mazur). The final congruence is stated in Proposition \ref{prop_main_congruence_L}. Let us also notice that a similar congruence formula has been obtained in a joint work of the first author \cite{Lecouturier_Wang} using Sharifi's conjecture and algebraic $K$-theory of cyclotomic fields. 

The second step, which perhaps is the most original part of this paper, is to transfer the congruence for the $L$-value from the first step to a congruence satisfied by a certain modular form of weight $\frac{3}{2}$ and level $\Gamma_1(4\ell^2)$. In some sense, our result is an Eisenstein and higher Eisenstein congruence at the level of modular forms of weight $\frac{3}{2}$. We rely crucially on a generalization of the Shimura correspondence due to Baruch--Mao in \cite{Mao}. The main results are Theorem \ref{prop_main_wt_3/2} and Corollary \ref{cor_main_wt_3/2}. The main technical difficulty is to take care of the Hecke operator at the place $2$, which requires going through the details of the construction of \cite{Mao}.

The third step is to use the Sturm bound to produce enough non-zero coefficients of our weight $\frac{3}{2}$ modular form modulo $p$. Our technique is similar to the one of Ono \cite{Ono} and Byeon \cite{Byeon}. Note that we need to assume that at least one coefficient is non-zero modulo $p$, which is the reason for the assumption $\mathcal{D}(\ell, p, Q) \neq \emptyset$ in Theorem \ref{main_thm_paper}.

\

\section*{Notation}
In this paper, $\ell$ and $p$ will denote two primes such that $p\mid (\ell-1)$. Since the group $(\Z/\ell\Z)^{\times}$ is cyclic of order $\ell-1$, we can fix a surjective group homomorphism
$$\log_{\ell, p} : (\Z/\ell\Z)^{\times} \rightarrow \Z/p\Z \text{.}$$

The letter $\Delta$ will be used to denote the discriminant of a real quadratic field, viewed inside $\R$. We let $h(\Delta)$ and $\epsilon_{\Delta}$ be the class number and the fundamental unit $>1$ of $\Q(\sqrt{\Delta})$ respectively. We denote by 
$$\chi_{\Delta} : (\Z/\Delta\Z)^{\times} \rightarrow \C^{\times}$$
the (primitive) quadratic character associated with $\Q(\sqrt{\Delta})$. 
Finally, we denote by $\mathcal{D}_{\ell}$ the set of discriminants of real quadratic fields in which the prime $\ell$ splits, \ie
\begin{equation}\label{eq_def_D_ell}
\mathcal{D}_{\ell} = \{\Delta>0 \text{ fundamental discriminant such that } \chi_{\Delta}(\ell)=1\} \text{.}
\end{equation}
Let $\Delta \in \mathcal{D}_{\ell}$.
If $\varepsilon$ is a unit of $\Q(\sqrt{\Delta})$ and $\l$ is a prime above $\ell$ in $K$, we define 
$$\log_{\l,p}  (\varepsilon):=\log_{\ell,p}  (\varepsilon \ {\rm modulo} \ \l)\in \Z/p\Z .$$ 
We observe that  the condition $\log_{\l,p}(\varepsilon)=0$ does not depend on the choice of $\l$ and of the discrete logarithm $\log_{\ell,p}$.

\

\section*{Acknowledgements and funding}

This project started in the summer 2023 when the second author visited the first author at the YMSC of Tsinghua University. The main research was carried out during subsequent visits of the second author to Westlake University (springs 2025 and 2026) and of the first author to University Marie and Louis Pasteur in summer 2025. The authors thank these institutions for their warm hospitality. The authors would also like to thank Farshid Hajir for discussions regarding the Cohen--Lenstra heuristics.

The first author was partially supported by a NSFC grant (grant number W2532004).
The second author was partially supported by the EIPHI Graduate School (contract “ANR-17-EURE-0002") and by the Bourgogne-Franche-Comté Region.

\

\section{Arithmetic background}\label{section1}

\subsection{Ray class group of real quadratic fields}

\subsubsection{Cyclic degree $p$ extensions tamely ramified}
 Let $K$ be a number field and $p\geq 2$ be a prime. 
A cyclic degree $p$ extension $L/K$ will be called a $\Z/p$-extension.
A prime $\l$ of $K$ is said to be \emph{tame} if its  absolute norm $N\l$ is congruent to $1$ modulo $p$. This congruence is necessary to have (possibly) some ramification at $\l$ in some $\Z/p$-extension of $K$.
In the following we are interested in tamely ramified $\Z/p$-extension $L/K$. As we will see, these extensions are managed by a certain governing field $M$.
Let 
$$V_K:=\{x\in K^\times \text{ such that } \forall \mathfrak{q} \text{ prime of }K \text{, }v_\q(x)\equiv 0 \ ({\rm modulo} \ p)\}$$ 
be the Selmer group of $K$. Here $v_\q$ is the normalized $\q$-valuation. It is well-known that  $V_K$ fits in the following exact sequence 
$$1 \longrightarrow  E_K/E_K^p \longrightarrow V_K/(K^\times)^p \longrightarrow \Cl_K[p] \longrightarrow 1,$$
where $E_K$ denotes the group of the units of $K$.
Set  $K':=K(\zeta_p)$, $M:=K'(\sqrt[p]{V_K})$ and $G:=\Gal(M/K')$.

Observe that a tame prime $\l$ of $K$ splits in $K'/K$, and is unramified in $M/K'$. Choose a prime $\Ll$ of $K'$ above $\l$, and denote by $\sigma_\Ll\in G$ its  Frobenius element in $M/K$. Observe also that if we take another prime $\Ll'$ over $\l$, then there exists $\lambda \in \Z/p\Z$ such that $\lambda \neq 0$ and $\sigma_\Ll=\lambda \sigma_{\Ll'}$ (we use additive notation).
From now on, for each tame prime $\l$ of $K$, we fix a Frobenius $\sigma_\Ll \in G$, and we note it $\sigma_\l$.
Let us recall the following result of Gras--Munnier \cite{GM} (see also \cite{HMR_CMJ}).

\begin{thm} \label{theo_GM}
There exist a $\Z/p$-extension $L/K$, exactly and totally ramified at tame primes $\{\l_1,...,\l_m\}$, if and only if there exist $a_1, ..., a_m \in (\Z/p\Z)^\times$ such that $$\sum_{i=0}^m a_i \sigma_{\l_i}=0 \in G\text{.}$$
\end{thm}

As easy consequence we get

\begin{cor}\label{coro_GM} Let $\l$ be a tame prime of $K$.
    There exists no $\Z/p$-extension $L/K$ unramified outside $\l$ if and only if $p\nmid \# \Cl_K$ and $\sigma_{\l} \neq 0 \in \Gal(M/K')$. 
\end{cor}

\begin{rem} \label{remark1}
For $p>2$, and $K=\Q$, $M=\Q(\zeta_p)$, $G=1$, and then for every prime $p| (\ell-1)$, there exists a $\Z/p$-extension $L/\Q$ exactly ramified at $\ell$ (this is a subextension of $\Q(\zeta_\ell)/\Q$). 
\end{rem}

\subsubsection{Tame Ray class group of real quadratic fields}
Let $\ell,p \geq 5$ be primes such that $p \mid (\ell- 1)$.

Let $\Delta \in \mathcal{D}_{\ell}$, $K=\Q(\sqrt{\Delta})$ and let $\l$ and $\l'$ be the prime ideals above $\ell$ in $K$.  By Remark \ref{remark1}, there exists a $\Z/p$-extension $L/K$ unramified outside~$\ell$. By Theorem \ref{theo_GM} there is no $\Z/p$-extension of $K$ unramified outside $\l$ if and only if $p\nmid h(\Delta)$ and $\sigma_\l \neq 0 \in \Gal(K(\sqrt[p]{\varepsilon_\Delta})/K')$. Thus, we get the following

\begin{prop} \label{prop_unramifiedoutsidel}
    There is no $\Z/p\Z$-extension of $K$ unramified outside $\l$, if and only if
    \begin{enumerate}
    \item[$(i)$]$p\nmid h(\Delta)$ and, 
    \item[$(ii)$]  $\varepsilon_\Delta$   is not a $p$th power modulo $\l$, or equivalently $\log_{\l, p}(\varepsilon_\Delta)\neq 0$.
    \end{enumerate}
\end{prop}

Observe that $(ii)$ does not depend on the choice of $\l|\ell$.

\begin{defn}
Given a fundamental discriminant $\Delta>0$, set 
$$a_{\Delta,\l,p}:=h(\Delta) \cdot \log_{\l, p}(\varepsilon_\Delta) \in \Z/p\Z \text{.}$$
Observe that $a_{\Delta,\l,p}\neq 0$ if and only if, there is no $\Z/p\Z$-extension of $K$ unramified outside $\l$.
\end{defn}

\begin{prop}\label{prop_link_h_ell_unit} With the notation above, we have $a_{\Delta,\l, p}\neq 0$ if and only if $p\nmid h^-_\ell(\Delta)$.
\end{prop}

\begin{proof} Assume first that $p\mid h(\Delta)$. Since $h(\Delta)=h^-(\Delta)$, then $p\mid h_\ell^-(\Delta)$, so the Proposition trivially holds in this case. From now on, assume that $p\nmid h(\Delta)$. 

Assume $\log_{\l, p}(\varepsilon_\Delta)=0$. By Proposition \ref{prop_unramifiedoutsidel}, there exists a $\Z/p$-extension $L/K$ exactly ramified at $\l$ (\ie ramified at $\l$ and unramified outside $\l)$. The action of $\Gal(K/\Q)$ on $L/K$ produces a $\Z/p$-extension $L'/K$ exactly ramified at $\l'$. Let $K_\ell^{ab}$ be the maximal abelian  $p$-extension of $K$ unramified outside $\ell$. Then $LL' \subset K_\ell^{ab}$, and by Global Class Field Theory, $\Cl_{K,\ell} \otimes_{\Z} \Z_p \simeq \Gal(K_\ell^{ab}/K)$. Moreover $\Gal(K/\Q)$ acts on $\Cl_{K,\ell}$, and this action is not trivial (it permutes $L$ and $L'$). Thus, $p\mid h_\ell^-(\Delta)$, and we conclude that $a_{\Delta,\l,p}=0$ implies $p\mid h_\ell^-(\Delta)$. 

Suppose now that $p\mid h_\ell^-(\Delta)$. Since $p\nmid h(\Delta)=h^-(\Delta)$, there exists a $\Z/p$-extension $L/K$ such that $\Gal(K/\Q)$ acts by $-1$ on $\Gal(L/K)$, and in particular the extension $L/K$ is exactly ramified at $\l$ and at~$\l'$. Moreover, by Remark \ref{remark1} there exists a $\Z/p$-extension $L_0/\Q$ exactly ramified at $\ell$. The $\Z/p$-extension $L_0K/K$ is exactly ramified at $\l$ and at $\l'$, and the Galois group $\Gal(K/\Q)$ acts trivially on $\Gal(L_0K/K)$. Hence $L_0K\neq L$. In the compositum $L_0KL/K$ let us consider the subfield $L_1/K$ fixed by the inertia at $\l'$. This is $\Z/p$-extension unramified outside $\l$, and since $p\nmid h(\Delta)$, $L_1/K$ is exactly ramified at $\l$.  By Proposition~\ref{prop_unramifiedoutsidel} (ii), we get $\log_{\l, p}(\varepsilon_\Delta)=0$.
\end{proof}

\begin{rem} When $K$ and $p$ are fixed and $p\nmid h(\Delta)$, the density of primes $\ell$ with $a_{\Delta,\l,p}=0$ is equal to $\frac{1}{p}$ by the Chebotarev density theorem. This follows directly from Theorem \ref{theo_GM}.
\end{rem}


\subsection{The case of Sophie Germain primes}

The goal of this section is to prove the existence of $\Delta \in {\mathcal D}_\ell$ such that $a_{\Delta,\l,p}\neq 0$ 
when $\ell=2p+1$ (\ie $p=\frac{\ell-1}{2}$ is a \emph{Sophie Germain prime}).

\subsubsection{Upper bounds for the class number of real quadratic fields}

We shall need an upper bound for $h(\Delta)$, where $\Delta$ is the discriminant of a real quadratic field. By \cite[Theorem (a)]{bound_clgp}, we have 
\begin{equation}\label{eq_naive_bound}
h(\Delta) \leq \frac{\sqrt{\Delta}}{2}.
\end{equation}
However, this bound (which is valid for all $\Delta$) will not be good enough for our purposes. To get a bound on $h(\Delta)$, it is enough using the class number formula to get an upper bound on $\abs{L(\chi_{\Delta},1)}$ and a lower bound on $\epsilon_\Delta$. To the best of the authors' knowledge, the best general available explicit upper bound is given by \cite[Theorem 2]{Ramaré_new_bound}: we have $\abs{L(\chi_{\Delta},1)} \leq \frac{9}{20}\cdot \log(\Delta)$ for $\Delta \geq 2\cdot 10^{49}$, but this would also not be enough for our purposes. 

Instead, we are going to use two explicit upper-bounds due to \cite{bound_clgp} and \cite{Ramaré_old_bound}, which are better and valid even for relatively small $\Delta$, but assume that $3$ is inert in $K$ (and $2$ ramifies in $K$ for one of them; let us also note that the result we need from \cite{bound_clgp} follows easily from the techniques of \cite{Louboutin}). This turns out to be suitable for Theorem \ref{thm_SophieGermain}. The upper-bounds we get as a corollary of \cite{bound_clgp} and \cite{Ramaré_old_bound} are the following.

\begin{prop}\label{thm_bound_classgp}
Assume that $3$ is inert in $K$. If $\Delta \geq 442368$, we have $h(\Delta) \leq \frac{\sqrt{\Delta}}{3}$. If in addition $2$ ramifies in $K$ and $\Delta > 4\cdot 2^6\cdot 6^9\cdot e^{24} = 6.83\ldots\cdot 10^{19}$, then $h(\Delta) \leq \frac{\sqrt{\Delta}}{6}$.
\end{prop}

\begin{proof}
We first check that the inequality $h(\Delta) \leq \frac{\sqrt{\Delta}}{3}$ holds for $\Delta \geq 442368$. Recall that by the class number formula, we have
$$h(\Delta) = \frac{\sqrt{\Delta}}{2\log(\epsilon_{\Delta})}\cdot L(\chi_{\Delta},1).$$
We have $\epsilon_{\Delta} = \frac{a+b\sqrt{\Delta}}{2}$ with $a,b\in \Z$. The assumption $\epsilon_{\Delta}>1$ implies easily $a,b>0$. In particular, we have $\epsilon_{\Delta} > \frac{\sqrt{\Delta}}{2}$ so $\log(\epsilon_{\Delta})>\frac{\log(\Delta)}{2}-\log(2)$. 

Using the assumption $\chi_{\Delta}(3)=-1$ (\ie $3$ is inert in $K$), a direct application of \cite[Corollary]{Ramaré_old_bound} with $q=\Delta$, $h=3$ and $k=1$ gives, for $\Delta \geq 25$:
$$\abs{L(\chi_{\Delta},1)} < \frac{1}{4}\cdot(\log(\Delta)+\log(12)).$$ 
(Let us note that the absolute values are actually superfluous, as $L(\chi_{\Delta},1)>0$ by the class number formula.)

We thus get
$$h(\Delta) < \frac{\sqrt{\Delta}}{4}\frac{\log(\Delta)+\log(12)}{\log(\Delta)-2\log(2)}.$$
For $\Delta \geq 4\cdot 48^3 = 442368$, we have 
$$\frac{\log(\Delta)+\log(12)}{\log(\Delta)-2\log(2)} = 1 + \frac{\log(48)}{\log(\Delta)-\log(4)} \leq \frac{4}{3}.$$
Therefore we get $h(\Delta) \leq \frac{\sqrt{\Delta}}{3}$, as wanted.

Let us now assume that $2$ ramifies in $K$. By \cite[Lemma 3]{bound_clgp}, we have 
$$\abs{L(\chi_{\Delta},1)} < \frac{1}{8}\cdot(\log(\Delta)+3\log(6)+8).$$ 
We thus get
$$h(\Delta) < \frac{\sqrt{\Delta}}{8}\frac{\log(\Delta)+3\log(6)+8}{\log(\Delta)-2\log(2)}.$$
For $\Delta \geq 4\cdot 2^6\cdot 6^9\cdot e^{24} = 6.83\ldots\cdot 10^{19} $, we have
$$\frac{\log(\Delta)+3\log(6)+8}{\log(\Delta)-2\log(2)} = 1 + \frac{\log(4)+3\log(6)+8}{\log(\Delta)-\log(4)} \leq \frac{4}{3}.$$
Therefore we get $h(\Delta) \leq \frac{\sqrt{\Delta}}{6}$, as wanted.
\end{proof}

\begin{rem}
One checks numerically that the inequality $h(\Delta) \leq \frac{\sqrt{\Delta}}{3}$ (with $3$ inert in $K$) holds for $\Delta < 442368$ unless $\Delta \in \{5,8\}$.
\end{rem}

\subsubsection{Explicit construction of certain real quadratic fields}
Recall that $p$ and $\ell$ are primes such that $p \mid (\ell-1)$. 

\begin{lem}\label{lem_a_b_primes}
Assume $p>2$. Let $a,b\in \Z$ and $q\neq \ell$ be an odd prime (possibly equal to $p$).
Suppose that 
\begin{enumerate}[label=(\roman*)]
    \item\label{lem_a_b_primes_i} $b=a^2\pm 1$,
    \item\label{lem_a_b_primes_ii} $\left(\dfrac{\ell^2+2a\ell +b}
{q}\right)=-1$,
\item\label{lem_a_b_primes_iii} $b\equiv c^2 \text{ (modulo } \ell\text{)}$ for some $c \in \Z$,
\item\label{lem_a_b_primes_iv} $\log_{\ell,p}(a+c) \neq 0$.
\end{enumerate}
Let $\Delta$ be the discriminant of the number field $K = \Q(\sqrt{\ell^2+2a\ell +b})$. Then we have $\Delta \in \mathcal{D}_{\ell}$, $q$ is inert in $K$ and $\log_{\l,p}(\epsilon_{\Delta}) \neq 0$ for any prime $\mathfrak{l}$ above $\ell$ in $K$.
\end{lem}

\begin{proof} Let $D = \ell^2+2a\ell +b = (\ell+a)^2 \pm 1$ and $K=\Q(\sqrt{D})$.
Note that if $D\leq  0$, then  either $D=-1$ and $a=-\ell$,  which contradicts  \ref{lem_a_b_primes_iv} using the  assumption $p>2$,  or $D=0$ which contradicts assumption \ref{lem_a_b_primes_ii}. Thus, we have $D>0$. We also have $K\neq \Q$ by \ref{lem_a_b_primes_ii}. Note that $\ell \nmid c$, as otherwise $a^2 \equiv \mp 1 \text{ (modulo }\ell\text{)}$, which would contradict \ref{lem_a_b_primes_iv} since $p>2$. Finally $\ell$ splits in $K$ by (iii) and the fact that $\ell \nmid c$. Therefore, we have proved that the discriminant $\Delta$ of $K$ satisfies $\Delta \in \mathcal{D}_{\ell}$.

Let $\mathfrak{l}$ be the prime above $\ell$ in $K$ such that $-c+\sqrt{D} \in \l$, and let $u=\ell + a +\sqrt{D}$, which belongs to the ring of integers of~$K$. By \ref{lem_a_b_primes_i}, we have $N_{K/\Q}(u)=a^2-b=\mp 1$. Hence, $u$ is a (nonzero) power of the fundamental unit $\epsilon_{\Delta}$ of~$K$. We have $u \equiv a+c \text{ (modulo }\mathfrak{l}\text{)}$, so by (iv) $\log_{\l,p}(u ) \neq 0$, which implies  $\log_{\l,p}(\epsilon_{\Delta}) \neq 0$. 
\end{proof}

\subsubsection{Sophie Germain primes}

Let us apply this construction in the case of Sophie Germain primes. 

\begin{thm}\label{thm_SophieGermain}
Assume $\ell \geq 11$ is such that $\frac{\ell-1}{2}$ is a Sophie Germain prime, \ie $\ell=2p+1$ with $p\geq 5$. Then there exists $\Delta \in \mathcal{D}_{\ell}$ such that $$ h(\Delta) \cdot \log_{\l,p}(\epsilon_{\Delta}) \neq 0$$ for any prime $\mathfrak{l}$ above $\ell$ in $K=\Q(\sqrt{\Delta})$ and $3$ is inert in $K/\Q$.
\end{thm}

\begin{proof}
We check numerically that the theorem holds for $p$ such that $p^2<2^6\cdot 6^9\cdot e^{24}$, \ie $p< 4.13\ldots \cdot 10^9$. More precisely, in this situation there are $11,912,302$ Sophie Germain primes, and for each one  we can find a field $K$ with discriminant $D \leq 236$ that satisfies the theorem. 
We can thus assume that $p^2\geq 2^6\cdot 6^9\cdot e^{24}$.

Let us apply the construction of Lemma \ref{lem_a_b_primes}. We first claim that there exists $a \in \{3,5,7\}$ such that $a^2-1$ is a square modulo $\ell$. Indeed, otherwise we would get
$$\left(\dfrac{8}{\ell}\right)=\left(\dfrac{24}{\ell}\right)=\left(\dfrac{48}{\ell}\right)=-1$$
which is clearly impossible. Now, choose $a \in \Z$ such that the following three conditions are satisfied:
\begin{itemize}
\item $a \equiv 3,5 \text{ or }7 \text{ (modulo }\ell\text{)}$ such that $a^2-1$ is a square modulo $\ell$ (this is possible by the above argument),
\item $2 \mid a$,
\item $a \equiv 1 \text{ (modulo }3 \text{)}$.
\end{itemize}
By the Chinese remainder theorem, we can find such an $a$ with $-4\ell < a < 2\ell$. Let $b=a^2-1$, and consider the number field $K=\Q(\sqrt{D})$ where $D=\ell^2+2a\ell+b$. Let $\Delta$ be the discriminant of $K$, which divides $D$ since $D\equiv 0 \text{ (modulo }4\text{)}$ (because $b$ is odd).

Note that the conditions (i) and (iii) of Lemma \ref{lem_a_b_primes}  are satisfied by construction. Furthermore, the condition $a \equiv 1 \text{ (modulo }3 \text{)}$ implies $D  \equiv -1 \text{ (modulo }3\text{)}$ (we use the fact that $\ell \equiv -1 \text{ (modulo }3\text{)}$ as $\ell=2p+1$ with $p>3$ prime). Thus, condition (ii) is satisfied for $q=3$.
Let us check that condition (iv) is also satisfied. Let $c$ such that $b\equiv c^2 \text{ (modulo }\ell\text{)}$. If $\log_{\ell,p}(a+c)=0$, then $a+c$ is a $p$th power modulo $\ell$. Since $\ell=2p+1$, this means that $a+c \equiv \pm 1 \text{ (modulo }\ell\text{)}$. We get $c \equiv -a\pm 1 \text{ (modulo }\ell\text{)}$, so $c^2 \equiv (a \pm 1)^2 \equiv a^2+1\pm 2a \text{ (modulo }\ell\text{)}$, so $a \equiv \pm 1 \text{ (modulo }\ell\text{)}$, which is impossible since $a \equiv 3,5 \text{ or }7 \text{ (modulo }\ell\text{)}$ and $\ell>3$.
By Lemma \ref{lem_a_b_primes}, we get that $\Delta \in \mathcal{D}_{\ell}$ (in particular, $K$ is a real quadratic field) and $\log_{\l,p}(\epsilon_\Delta) \neq 0$. 

\smallskip

Let us now prove that $p \nmid h(\Delta)$. For the sake of a contradiction, assume $p \mid h(\Delta)$ and in particular $h(\Delta) \geq p$. Since $\Delta \geq 4h(\Delta)^2$ by (\ref{eq_naive_bound}), we get $\Delta \geq 4p^2 \geq 4\cdot 2^6\cdot 6^9\cdot e^{24}$.  

Assume first that $\Delta$ is odd. Then $\Delta$ divides $\frac{D}{4} = \frac{(\ell+a)^2-1}{4} = \frac{(\ell+a-1)}{2}\cdot \frac{(\ell+a+1)}{2}$. Since visibly $\frac{D}{4}$ is even, actually $\Delta$ divides $\frac{D}{8}$. Since $\Delta \geq 4 \cdot 2^6\cdot 6^9\cdot e^{24} \geq 442368$, by the first inequality of Proposition \ref{thm_bound_classgp} we get $h(\Delta) \leq \frac{1}{3}\cdot \sqrt{\frac{D}{8}}= \frac{1}{6\sqrt{2}}\cdot \sqrt{D}$. Since $D = (\ell+a)^2-1$, we have $h(\Delta)<\frac{1}{6\sqrt{2}}\cdot \abs{\ell+a}$, and since $-4\ell < a < 2\ell$, we get $h(\Delta)<\frac{1}{2\sqrt{2}}\ell = \frac{1}{\sqrt{2}}p+\frac{1}{2\sqrt{2}} < p$, which is a contradiction. This proves that $p\nmid h(\Delta)$.

Assume now that $\Delta$ is even. Since $\Delta \geq 4 \cdot 2^6\cdot 6^9\cdot e^{24}$, $2$ ramifies in $K$ and $3$ is inert in $K$, the second inequality of Proposition \ref{thm_bound_classgp} yields $h(\Delta) \leq \frac{\sqrt{\Delta}}{6} \leq \frac{\sqrt{D}}{6}$. 
Since $D = (\ell+a)^2-1$, we have $h(\Delta)<\frac{1}{6}\cdot \abs{\ell+a}$. Recall that $-4\ell < a < 2\ell$, so $\abs{a+\ell}\leq 3\ell-1$. Actually, we claim that $\abs{a+\ell}\leq 3\ell-3$. Otherwise, we would have $a \in \{2\ell-1, 2\ell-2, -4\ell+1, -4\ell+2\}$, which would imply $a\equiv \pm 1, \pm 2 \text{ (modulo }\ell\text{)}$, \ie $a^2 \equiv 1, 4 \text{ (modulo } \ell\text{)}$, which is impossible as $a\equiv 3, 5, 7 \text{ (modulo }\ell\text{)}$ (by construction) and $\ell > 7$. We get $h(\Delta) < \frac{1}{6}\cdot (3\ell-3) = p$. This is a contradiction.
\end{proof}


\subsection{Heuristic observation}\label{section_heuristic}

In our context, it is natural to ask the following question: \textit{
Given two primes $p$ and $\ell$ such $p\mid (\ell-1)$, how often  $a_{\Delta,\l,p}\neq 0$
 as  $\Delta$ ranges over $\Dd_{\ell}$?}

Let us then conduct some numerical simulations.
For $p\mid \ell -1$, $X \geq 2$, set $$N_{\ell}(X)=\# \{0<\Delta\leq X \text{ such that }\Delta \in \Dd_\ell\}$$ 
and
$$
M_{p,\ell}(X)=\#\{0<\Delta\leq X \text{ such that } \Delta \in \Dd_\ell \text{ and } p\nmid h_\ell^-(\Delta)\}\cdot
$$

Using PARI/GP \cite{PARI}, we get, for $X=10^8$:
$$M_{3,7}/N_7 \approx 0.5788; \  M_{3,13}/N_{13} \approx 0.5779; $$
$$M_{5,11}/N_{11} \approx 0.7630; \ M_{5,101}/N_{101} \approx 0.7625;$$
$$M_{7,29}/N_{29}\approx 0.8380; \ M_{7,43}/N_{43} \approx 0.8381.$$
These results should be compared with the quantities   $\eta^1(3) \approx 0.5925$, 
$\eta^1(5)\approx 0.7680$, and $\eta^1(7) \approx 0.8396$.

As we noted in the introduction, we suspect that the quantity $M_{p,\ell}(X)/N_{\ell}(X)$ converges to $\eta^1(p)$ as $X$ tends to infinity. The interpretation could be as follows. 
First of all, as with the Cohen--Lenstra heuristics for the $p$-Sylow subgroup of the class group of quadratic fields, everything takes place in the minus part. Next, let us recall the central idea that allows us to move from a density of indivisibility for class groups of imaginary quadratic fields to a density of indivisibility for class groups  of real quadratic fields. It is based on the following principle: For a family of abelian groups (corresponding to the class group of imaginary quadratic fields), there is an associated family of groups in which each group is the quotient group of a randomly chosen element of a group from the original family (corresponding to the class group of real quadratic fields).
We then move from $\eta^1(p)$ to $\eta^2(p)$. Observe now that in the case of real quadratic fields, we move from the ray class group of modulus $\l$ to the class group by taking the  quotient with respect to the ramification subgroup at $\l$, which is cyclic in this case. Within the heuristic framework, this can also be viewed as a quotient by some arbitrary  element. Thus, we return from $\eta^2(p)$ to $\eta^1(p)$.

\section{Eisenstein congruences in weight $2$}\label{section_Eisenstein_congruences}

\subsection*{Notation}
We keep the same notation as above. In this section, we assume $p\geq 5$. 
Recall that $\Delta$ denotes the discriminant of a real quadratic field $K=\Q(\sqrt{\Delta})$.
We fix an embedding $K \hookrightarrow \R$. We let $h^+(\Delta)$ be the cardinality of the \emph{narrow} class group $\Cl^+(K)$, and $\epsilon^+_{\Delta}$ be the smallest \emph{totally positive} unit of $K$ which is $>1$. Note that here the $+$ in the exponent is not related to the action of $\Gal(K/\Q)$ (contrary to the notation $h^-_{\ell}(\Delta)$).
Note that $h^+(\Delta)=2h(\Delta)$ or $h(\Delta)$, depending on whether $\epsilon^+_{\Delta} = \epsilon_{\Delta}$ or $\epsilon^+_{\Delta} = \epsilon_{\Delta}^2$ respectively. In any case, we have $(\epsilon^+_{\Delta})^{h^+(\Delta)} = \epsilon_{\Delta}^{2h(\Delta)}$.

\subsection{Review of Mazur's Eisenstein ideal}\label{section_review_Eisenstein}
In this paragraph, we briefly recall some well-known facts about the Eisenstein ideal theory as developed by Mazur in his seminal paper \cite{Mazur_Eisenstein}. Let us note that, following Mazur, most papers on this topic have used the notation $N$ for our prime $\ell$. Since the focus of our paper is classical algebraic number theory, we have decided to use the more common letter $\ell$ to denote that prime.

Let $\mathbb{T}^0$ be Hecke algebra over $\Z$ acting on the space $S_2(\Gamma_0(\ell))$ of weight $2$ cuspforms of level $\Gamma_0(\ell)$. It is generated by the Hecke operators $T_q$ for primes $q\neq \ell$ as well as the operator $U_{\ell}$. Mazur defined the \emph{Eisenstein ideal} $I \subset \mathbb{T}^0$ as the ideal generated by the operators $T_q-q-1$ for primes $q \neq \ell$ and by $U_{\ell}-1$. Mazur proved in \cite[Proposition II.9.7]{Mazur_Eisenstein} that $\mathbb{T}^0/I$ is cyclic of order the numerator of $\frac{\ell-1}{12}$. In particular, since $p\geq 5$ divides $\ell-1$, there is a ring homomorphism 
\begin{equation}\label{eq_T_Z/p}
\mathbb{T}^0 \rightarrow \Z/p\Z
\end{equation}
whose kernel is the maximal ideal $\mathfrak{P} = I+p\mathbb{T}^0$. 

Mazur constructed in \cite[Proposition II.18.9 and Theorem II.18.10]{Mazur_Eisenstein} an \emph{explicit} group isomorphism
\begin{equation}\label{eq_Mazur_I/I^2}
I/I^2 \xrightarrow{\sim} (\Z/\ell\Z)^{\times}/\mu_{12} \text{,}
\end{equation}
where $\mu_{12}$ is the subgroup of elements of order dividing $12$ in $(\Z/\ell\Z)^{\times}$. In particular, if we denote by $\mathbb{T}^0_I$ the $I$-adic completion of $\mathbb{T}^0$, then the ideal $I\cdot \mathbb{T}^0_I$ is principal. Similarly, denoting by $\mathbb{T}^0_{\mathfrak{P}}$ the $\mathfrak{P}$-adic completion of $\mathbb{T}^0$, the ideal $I\cdot \mathbb{T}^0_{\mathfrak{P}}$ of $\mathbb{T}^0_{\mathfrak{P}}$ is principal. 

Consider the first singular homology group $H_1(X_0(\ell), \Z)$ of the compact modular curve $X_0(\ell)$. It carries a natural action of $\mathbb{T}^0$ as well as the complex conjugation $c$ (induced by the map $x+iy\mapsto -x+iy$ on the upper-half plane). Following the notation of \cite[II.18]{Mazur_Eisenstein}, we denote by $H_+$ the subgroup of $H_1(X_0(\ell), \Z)$ fixed by $c$. By \cite[Proposition 5]{Merel_accouplement}, $H_1(X_0(\ell), \Z)$ is \emph{acyclic} for the action of $c$, \ie $H_+ = (1+c)\cdot H_1(X_0(\ell), \Z)$. If $\alpha$ and $\beta$ are elements in the completed upper-plane (\ie elements of $\mathbf{P}^1(\Q)$ or of the upper-half plane), we denote by $\{\alpha, \beta\}$ the relative homology class of the image in $X_0(\ell)$ of the geodesic path between $\alpha$ and $\beta$. We have $\{\alpha, \beta\} \in H_1(X_0(\ell), \Z)$ if the images of $\alpha$ and $\beta$ in $X_0(\ell)$ coincide.

Mazur's homomorphism (\ref{eq_Mazur_I/I^2}) can be conveniently described using $H_+$, using a slight modification of \cite[II.18]{Mazur_Eisenstein} relying on the acyclicity of $H_1(X_0(\ell), \Z)$ for $c$. This modification is due to Merel \cite[\S 4.1]{Merel_accouplement}, who used the subgroup $U$ of $12$th powers of $(\Z/\ell\Z)^{\times}$ instead of our quotient $(\Z/\ell\Z)^{\times}/\mu_{12}$ (the two groups being canonanically isomorphic).

Namely, \cite[Proposition II.18.8]{Mazur_Eisenstein} implies that there is a group isomorphism
\begin{equation}\label{eq_Mazur_H/I}
H_+/I\cdot H_+ \xrightarrow{\sim} (\Z/\ell\Z)^{\times}/\mu_{12} 
\end{equation}
sending $(1+c)\cdot \{z_0, \gamma\cdot z_0\}$ to $d$ modulo $\ell$, where $\gamma = \begin{pmatrix} * & * \\ * & d \end{pmatrix} \in \Gamma_0(\ell)$ and $z_0$ is any point in the completed upper-half plane, \ie an element of $\mathbf{P}^1(\Q)$ or of the upper-half plane (the element $ \{z_0, \gamma\cdot z_0\}$ of $H_1(X_0(\ell), \Z)$ does not depend on the choice of $z_0$). 

Following \cite[II.18 Definition p. 137]{Mazur_Eisenstein}, consider the \emph{winding homomorphism}
\begin{equation}\label{eq_winding_hom}
e_+ : I \rightarrow H_+
\end{equation}
defined by sending $\eta \in I$ to $\eta\cdot \{0, \infty\}$.  Let $(H_+)_I$ be the $I$-adic completion of $H_+$, which is canonically identified with $H_+ \otimes_{\mathbb{T}^0} \mathbb{T}^0_I$. By \cite[Theorem II.18.10]{Mazur_Eisenstein}, $e_+$ induces an \emph{isomorphism} 
$$e_+ : I \cdot \mathbb{T}^0_I \rightarrow (H_+)_I \text{.}$$
In particular, $e_+$ induces a group isomorphism $I/I^2 \xrightarrow{\sim} H_+/I\cdot H_+$, and the map (\ref{eq_Mazur_I/I^2}) is then the composition of (\ref{eq_Mazur_H/I}) with that latter isomorphism.

\subsection{Eisenstein congruence and periods}\label{section_periods}

The ring homomorphism (\ref{eq_T_Z/p}) corresponds to the existence of a newform $F = \sum_{n\geq 1} a_n(F)q^n\in S_2(\Gamma_0(\ell))$ such that, for all primes $q \neq \ell$, we have $a_q(F) \equiv q+1 \text{ (modulo } \mathfrak{P}_F \text{)}$, and $a_{\ell}(F) \equiv 1 \text{ (modulo } \mathfrak{P}_F \text{)}$ (actually, it is known that $a_{\ell}(F)=1$). Here, $\mathfrak{P}_F$ is a certain maximal ideal above $p$ in the ring of coefficients $\mathcal{O}_F := \Z[a_n(F)]_{n\geq 1}$ of $F$, which an order in the ring of integers of the Hecke field $K_F := \Q(a_n(F), n\geq 1) \subset \C$. Such a form may not be unique in general.

Equivalently, we have an \emph{Eisenstein congruence}
\begin{equation}\label{eq_Eis_congruence}
F \equiv E_{2, \ell} \text{ (modulo }\mathfrak{P}_F \text{)}
\end{equation}
where 
$$E_{2,\ell} = \frac{\ell-1}{24}+\sum_{n\geq 1} \left(\sum_{d \mid n \atop (d,\ell)=1} d \right)q^n$$
is the (unique) Eisenstein series of weight $2$ and level $\Gamma_0(\ell)$.  In the rest of this paper, we shall fix a newform $F$ and a maximal ideal $\mathfrak{P}_F \subset \mathcal{O}_F$ above $p$ satisfying the Eisenstein congruence (\ref{eq_Eis_congruence}). Note that there is a surjective ring homomorphism $$\varphi_F : \mathbb{T}^0 \rightarrow \mathcal{O}_F$$ and that we have $\mathfrak{P}_F = \varphi_F(\mathfrak{P})$. Let $\mathfrak{a}_F = \ker(\varphi_F)$ and $I_F = \varphi_F(I)$. Note that we have $\mathfrak{a}_F \subset \mathfrak{P}$ and $I_F \subseteq \mathfrak{P}_F$.

Let $\mathcal{O}_{\mathfrak{P}_F}$ be the $\mathfrak{P}_F$-adic completion of $\mathcal{O}_F$. Note that $\mathcal{O}_{\mathfrak{P}_F}$ is not \emph{a priori} a DVR, as $\mathcal{O}_F$ may be a strict order. The map $\varphi_F$ induces a surjective homomorphism of $\Z_p$-algebras 
\begin{equation}\label{eq_phi_F_P}
\varphi_{F, \mathfrak{P}} : \mathbb{T}_{\mathfrak{P}}^0 \rightarrow \mathcal{O}_{\mathfrak{P}_F}
\end{equation}
sending $I\cdot \mathbb{T}^0_{\mathfrak{P}}$ to $I_F \cdot \mathcal{O}_{\mathfrak{P}_F}$, and therefore induces a surjective group homomorphism 
$$I\cdot \mathbb{T}^0_{\mathfrak{P}}/I^2\cdot \mathbb{T}^0_{\mathfrak{P}} \rightarrow I_F \cdot \mathcal{O}_{\mathfrak{P}_F}/I_F^2 \cdot \mathcal{O}_{\mathfrak{P}_F} \text{.}$$

Note that $$I\cdot \mathbb{T}^0_{\mathfrak{P}}/I^2\cdot \mathbb{T}^0_{\mathfrak{P}} = (I/I^2)\otimes_{\Z} \Z_p \text{,}$$
and that we have a surjective group homomorphism $(I/I^2)\otimes_{\Z} \Z_p \rightarrow \Z/p\Z$ given by composing $\log_{\ell,p}$ with~(\ref{eq_Mazur_I/I^2}). Thus, our choice of $\log_{\ell,p}$ induces a surjective group homomorhism 
\begin{equation}\label{eq_I/I^2_p}
I\cdot \mathbb{T}^0_{\mathfrak{P}}/I^2\cdot \mathbb{T}^0_{\mathfrak{P}} \rightarrow \Z/p\Z \text{.}
\end{equation}

\begin{lem}\label{lemma_I/I^2_F}
The map (\ref{eq_I/I^2_p}) factors through the map $$I\cdot \mathbb{T}^0_{\mathfrak{P}}/I^2\cdot \mathbb{T}^0_{\mathfrak{P}} \rightarrow I_F \cdot \mathcal{O}_{\mathfrak{P}_F}/I_F^2 \cdot \mathcal{O}_{\mathfrak{P}_F} \text{.}$$
Thus, we have a surjective group homomorphism
\begin{equation}\label{eq_I_F/I_F^2}
I_F \cdot \mathcal{O}_{\mathfrak{P}_F}/I_F^2 \cdot \mathcal{O}_{\mathfrak{P}_F} \rightarrow \Z/p\Z
\end{equation}
depending only on the choice of $\log_{\ell, p}$.
\end{lem}
\begin{proof}
Since $I\cdot \mathbb{T}^0_{\mathfrak{P}}/I^2\cdot \mathbb{T}^0_{\mathfrak{P}} $ is a cyclic $p$-group, it suffices to prove that $I_F \cdot \mathcal{O}_{\mathfrak{P}_F}/I_F^2 \cdot \mathcal{O}_{\mathfrak{P}_F} $ is non-trivial. Otherwise, we would have $I_F\cdot \mathcal{O}_{\mathfrak{P}_F} = I_F^2\cdot \mathcal{O}_{\mathfrak{P}_F}$, so $I_F\cdot \mathcal{O}_{\mathfrak{P}_F} = \mathfrak{P}_F\cdot I_F \cdot \mathcal{O}_{\mathfrak{P}_F}$, and by Nakayama's lemma $I_F\cdot \mathcal{O}_{\mathfrak{P}_F}=0$ which is impossible.
\end{proof}

We shall need the following result in order to prove our congruence formula for $L$-values. We use the well-known fact that the $L$-value $L(F,1) = \int_{\{0,\infty\}} 2i\pi F(z)dz$ is non-zero (this is a consequence of \cite[Theorem II.18.10]{Mazur_Eisenstein}).

\begin{prop}\label{prop_period_F}
Consider the group homomorphism $p_F : H_+ \rightarrow \C$ defined by
$$p_F(x) = \frac{\int_{x}2i\pi F(z)dz}{L(F,1)} \text{.}$$
(This is well-defined since, as recalled above, we have $L(F,1) \neq 0$.)
\begin{enumerate}[label=(\roman*)]
\item\label{prop_period_F_i} We have $p_F(H_+) \subseteq K_F$, the latter being naturally a subfield of the fraction field of $\mathcal{O}_{\mathfrak{P}_F}$. Furthermore, we actually have $p_F(H_+) \subseteq I_F\cdot \mathcal{O}_{\mathfrak{P}_F}$. In particular, one gets a natural map
$$H_+ \rightarrow I_F\cdot \mathcal{O}_{\mathfrak{P}_F}/I_F^2\cdot \mathcal{O}_{\mathfrak{P}_F}$$
whose composition with the map (\ref{eq_I_F/I_F^2}) of Lemma \ref{lemma_I/I^2_F} yields a surjective group homomorphism
\begin{equation}\label{eq_alpha_H^+}
\alpha_F : H_+ \rightarrow \Z/p\Z \text{,}
\end{equation}
depending only on the choice of $\log_{\ell, p}$ and $F$, and whose kernel contains $I$ and $\mathfrak{a}_F$.
\item\label{prop_period_F_ii} For any $z_0$ in the completed upper-half plane and $\gamma = \begin{pmatrix} * & * \\ * & d \end{pmatrix} \in \Gamma_0(\ell)$, we have 
$$\alpha_F((1+c)\cdot \{z_0, \gamma\cdot z_0\}) = \log_{\ell, p}(d) \text{.}$$
In particular, $\alpha_F$ depends only on the choice of $\log_{\ell, p}$, and not on the choice of $F$ satisfying the Eisenstein congruence (\ref{eq_Eis_congruence}).
\end{enumerate}
\end{prop}
\begin{proof}
It will be convenient to extend slightly the definition of $p_F(x)$ to the \emph{Manin-Drinfeld modification} $\tilde{H}_+$ of $H_+$. Consider the relative homology group $H_1(X_0(\ell), \cusps, \Z)$, where $\cusps$ is the finite set of cusps of $X_0(\ell)$, namely the cusps $0$ and $\infty$ (since $\ell$ is prime). Let $\mathbb{T}$ be the Hecke algebra (over $\Z$) acting faithfully on $H_1(X_0(\ell), \cusps, \Z)$. Note that $\mathbb{T}^0$ is a quotient of $\mathbb{T}$, and that we have a $\mathbb{T}$-equivariant short exact sequence
$$0 \rightarrow H_1(X_0(\ell), \Z) \rightarrow H_1(X_0(\ell), \cusps, \Z) \rightarrow \Z[\cusps]^0 \rightarrow 0$$
where $\Z[\cusps]^0$ is the group of degree zero divisors supported on the set of $\cusps$, which is just $\Z$ since there are two cusps. The action of $\mathbb{T}$ on $\Z[\cusps]^0$ factors through the Eisenstein ideal of $\mathbb{T}$. Furthermore, there is a canonical $\mathbb{T}$-equivariant projection of $H_1(X_0(\ell), \cusps, \Z)$ onto its subspace $H_1(X_0(\ell), \Z)$ after inverting the numerator $n$ of $\frac{\ell-1}{12}$. This projection $p : H_1(X_0(\ell), \cusps, \Z) \rightarrow H_1(X_0(\ell), \Z[\frac{1}{n}])$, called the \emph{Manin--Drinfeld retraction}, is characterized by sending the modular symbol $\{0,\infty\}$ to the unique element 
$$e \in H_1(X_0(\ell), \Z[\frac{1}{n}])$$ (called the \emph{winding element}, \cf \cite[II.18 Definition p. 136]{Mazur_Eisenstein}) satisfying, for all cusp form $f \in S_2(\Gamma_0(\ell))$:
$$\int_e f(z)dz = \int_{\{0, \infty\}} f(z)dz \text{.}$$
We denote by $\tilde{H}$ the image of $H_1(X_0(\ell), \cusps, \Z)$ in $H_1(X_0(\ell), \Z[\frac{1}{n}])$ by the Manin--Drinfeld retraction $p$, and by $\tilde{H}_+$ the subspace of $\tilde{H}$ fixed by the complex conjugation. Note that we have an inclusion $H_+ \subset \tilde{H}_+$ and a canonical group isomorphism $\tilde{H}_+/H_+ \xrightarrow{\sim} \Z/n\Z$ sending the image of $e$ to $1$. Furthermore, $\mathbb{T}^0$ acts on $\tilde{H}_+$ and we have $I\cdot \tilde{H}_+ \subset H_+$ (\cf \cite[Lemma II.18.6]{Mazur_Eisenstein}). 

Our map $p_F : H_+ \rightarrow \C$ extends in a natural way to a map $\tilde{p}_F : \tilde{H}_+\rightarrow \C$ given by the same formula, \ie 
$$\tilde{p}_F(x) = \frac{\int_{x}2i\pi F(z)dz}{L(F,1)} \text{.}$$
Note that, by construction, we have $\tilde{p}_F(e)=1$. We are now ready to prove the proposition.

The fact that $\tilde{p}_F$ takes values in the Hecke field $K_F$ of $F$ is well-known, \cf \eg \cite[Proposition 5.11]{Pasol_Popa}. Let $\tilde{M}:=\tilde{p}_F(\tilde{H}_+)$ and $M:=p_F(H_+)$. We have $M \subset \tilde{M}\subset K_F$. If $m$ is a positive integer, then we have, for all $x \in \tilde{H}_+$: 
$$\tilde{p}_F(T_mx) = a_m(F)\cdot \tilde{p}_F(x) \text{.}$$
This proves that the kernel of $\tilde{p}_F$ (and of $p_F$) contain $\mathfrak{a}_F$, and furthermore that both $\tilde{M}$ and $M$ are finitely generated sub-$\mathcal{O}_F$-modules of $K_F$. Note that $\tilde{M}$ contains $\mathcal{O}_F$ since $\tilde{p}_F(e)=1$.

Denote by $\tilde{M}_{\mathfrak{P}_F}$ (resp. $M_{\mathfrak{P}_F}$) the completion of $\tilde{M}$ (resp. $M$) at $\mathfrak{P}_F$. We have a natural inclusion $\mathcal{O}_{\mathfrak{P}_F} \subseteq \tilde{M}_{\mathfrak{P}_F}$. The map $\tilde{p}_F : \tilde{H}_+ \rightarrow \tilde{M}$ induces, after completion at $\mathfrak{P}$ on the left and at $\mathfrak{P}_F$ on the right, a map
$$(\tilde{H}_+)_{\mathfrak{P}} \rightarrow M_{\mathfrak{P}_F} \text{.}$$

Since the winding homomorphism (\ref{eq_winding_hom}) is an isomorphism after completion at $\mathfrak{P}$, we easily see that we have $(\tilde{H}_+)_{\mathfrak{P}} = \mathbb{T}_{\mathfrak{P}}^0\cdot e$ and $(H_+)_{\mathfrak{P}} = I\cdot \mathbb{T}_{\mathfrak{P}}^0\cdot e$. Thus, we have $\tilde{M}_{\mathfrak{P}_F} = \mathcal{O}_{\mathfrak{P}_F}$ and $M_{\mathfrak{P}_F} = I_F\cdot \mathcal{O}_{\mathfrak{P}_F}$. This proves point \ref{prop_period_F_i}. 

Let us now prove point \ref{prop_period_F_ii}. Let $\gamma = \begin{pmatrix} * & * \\ * & d \end{pmatrix} \in \Gamma_0(\ell)$ and $z_0$ in the completed upper-half plane. Let $x$ be the image of $(1+c)\cdot \{z_0, \gamma\cdot z_0\}$ in $(H_+)_{\mathfrak{P}}$. 
By \ref{prop_period_F_i}, we have $p_F(x) \in I_F\cdot \mathcal{O}_{\mathfrak{P}_F}$, and more precisely we have $p_F(x) = \varphi_{F, \mathfrak{P}}(\eta)$ where $\eta \in I\cdot \mathbb{T}_{\mathfrak{P}}^0$ is such that $x = \eta\cdot e$ and $\varphi_{F, \mathfrak{P}}$ is defined in (\ref{eq_phi_F_P}). By construction, $\alpha_F((1+c)\cdot \{z_0, \gamma\cdot z_0\}) \in \Z/p\Z$ is the image of $\eta$ in $\Z/p\Z$ via (\ref{eq_I/I^2_p}), which as we recalled at the end of \S \ref{section_review_Eisenstein} is equal to $\log_{\ell,p}(d)$.
\end{proof}

\subsection{Popa's formula and congruence for the $L$-value of real quadratic twists}

In this paragraph, we recall some consequences of a formula of Waldspurger \cite{Waldspurger_twist} for the twisted $L$-value $L(F,\chi_{\Delta},1)$, where $F \in S_2(\Gamma_0(\ell))$ is a newform and $\Delta \in \mathcal{D}_{\ell}$ (\ie such that $\chi_{\Delta}(\ell)=1$).

The explicit form of the formula we shall need is due to Popa \cite{Popa_L}, and is expressed in terms of \emph{Heegner cycles}. The main result of this paragraph is Proposition \ref{prop_main_congruence_L} below, which as far as we know is new, but let us emphasize that it is very similar to results obtained in \cite[Proposition 5.7]{DHRV} and by the first author in \cite[\S 3.3 (26)]{Lecouturier_HV}, and indeed our proof relies on similar techniques (namely Heegner cycles). A very similar result was also obtained in \cite{Lecouturier_Wang} by completely different methods. 

We first recall Popa's formula, following \cite[\S 6]{Popa_L}. Let us fix a choice of a square root $\delta_{\ell}$ of $\Delta$ modulo $4\ell$, \ie $\delta_{\ell} \in \Z$ (well-defined modulo $2\ell$) such that
$$\delta_{\ell}^2 \equiv \Delta \text{ (modulo }4\ell\text{).}$$
This choice determines an ideal $\mathfrak{l}$ of $\mathcal{O}_K$ above $\ell$, characterized by the property
\begin{equation}\label{eq_def_mathfrak_l}
\sqrt{\Delta} + \delta_{\ell} \in \mathfrak{l} \text{.}
\end{equation}

Let us recall briefly the notion of Heegner cycles (\cf \cite[\S 6.2 and 6.3]{Popa_L}). Recall that a $\Q$-algebra embedding $\Psi : K \rightarrow M_2(\Q)$ given by 
\begin{equation}\label{eq_def_opt_emb}
\Psi(\sqrt{\Delta}) = \begin{pmatrix} a & b \\ c & -a \end{pmatrix}
\end{equation}
is called \emph{optimal of level $\ell$} if $\Psi(K) \cap M_0(\ell) = \Psi(\mathcal{O}_K)$ and \emph{oriented} (with respect to our choice of $\delta_{\ell}$) if 
\begin{equation}\label{eq_congruence_a}
a \equiv \delta_{\ell} \text{ (modulo }2\ell\text{).}
\end{equation}
Here, we have denoted by $M_0(\ell)$ the Eichler order of level $\ell$ in $M_2(\Z)$ consisting of matrices which are upper-triangular modulo $\ell$. 

The group $\Gamma_0(\ell)$ acts (by conjugation) on the right on the set $\mathcal{E}_{\ell}$ of oriented optimal embeddings of level~$\ell$, and we have a natural bijection (\cf \cite[Proposition 6.2.1]{Popa_L})
\begin{equation}\label{eq_bijection_opt_emb_cl}
\mathcal{E}_{\ell}/\Gamma_0(\ell) \xrightarrow{\sim} \Cl^+(K) \text{.}
\end{equation}
In particular, the set $\mathcal{E}_{\ell}/\Gamma_0(\ell)$ is finite of cardinality $h_K^+$. 
For $\Psi \in \mathcal{E}_{\ell}$, let $$M_{\Psi} := \Psi(\epsilon_K^+) \in \Gamma_0(\ell) \text{.}$$
By construction of $\Psi$, we have $M_{\Psi} = \begin{pmatrix} * & * \\ * & d_{\Psi} \end{pmatrix}$ with $d_{\Psi}=m-na$ where $a$ is as in (\ref{eq_def_opt_emb}) and $m,n \in \frac{1}{2} \Z$ are such that $\epsilon_K^+ = m+n\sqrt{\Delta}$. It follows from (\ref{eq_congruence_a}) and our choice of $\mathfrak{l}$ in (\ref{eq_def_mathfrak_l}) that we have
\begin{equation}\label{eq_congruence_Heegner_cycle}
d_{\Psi} \equiv \epsilon_K^+ \text{ (modulo }\mathfrak{l}\text{).}
\end{equation}

For any $z_0$ in the completed upper-half plane, we thus get a homology class (independent of $z_0$), called a \emph{Heegner cycle}:
$$\gamma_{\Psi}:=\{z_0, M_{\Psi}z_0\} \in H_1(X_0(\ell), \Z) \text{.}$$

As noted in \cite[p.860]{Popa_L}, there is an involution $\Psi \mapsto \Psi^*$ on oriented optimal embeddings of level $\ell$ which preserves $\Gamma_0(\ell)$-equivalence and with the property that $\gamma_{\Psi^*} = c\cdot \gamma_{\Psi}$ (where $c$ is, as above, the complex conjugation acting on $H_1(X_0(\ell), \Z)$). We conclude that  $\sum_{\Psi \in \mathcal{E}_{\ell}/\Gamma_0(\ell)} \gamma_{\Psi}$ is fixed by $c$, \ie we have
\begin{equation}\label{eq_complex_conjug_Heegner}
2\cdot \sum_{\Psi \in \mathcal{E}_{\ell}/\Gamma_0(\ell)} \gamma_{\Psi} =  \sum_{\Psi \in \mathcal{E}_{\ell}/\Gamma_0(\ell)} (1+c)\cdot\gamma_{\Psi} \text{.}
\end{equation}

Let us now state Popa's formula. Let $F\in S_2(\Gamma_0(\ell))$ be a newform. As usual, denote by $L(F/K, s)$ the $L$-function of $F$ based change to $K$. We have the well-known factorization:
\begin{equation}\label{eq_factorization_L}
L(F/K,s) = L(F,s)L(F, \chi_{\Delta}, s)
\end{equation}
We then have, by \cite[Theorem 6.3.1]{Popa_L}:
$$
L(F/K,1) = \frac{1}{\sqrt{\Delta}}\abs{\sum_{\Psi \in \mathcal{E}_{\ell}/\Gamma_0(\ell)} \int_{\gamma_{\Psi}} 2i\pi F(z)dz }^2 \text{.}
$$

Using (\ref{eq_complex_conjug_Heegner}) and (\ref{eq_factorization_L}), if $L(F,1)\neq 0$ then one can rewrite Popa's formula as
\begin{equation}\label{eq_Popa_formula}
\sqrt{\Delta}\cdot \frac{L(F, \chi_{\Delta},1)}{L(F,1)} =  \left(\frac{1}{2}\sum_{\Psi \in \mathcal{E}_{\ell}/\Gamma_0(\ell)} \frac{\int_{(1+c)\cdot \gamma_{\Psi}} 2i\pi F(z)dz}{L(F,1)} \right)^2 \text{.}
\end{equation}

Note that we have removed the absolute values, because the complex number
$$\frac{\int_{(1+c)\cdot \gamma_{\Psi}} 2i\pi F(z)dz}{L(F,1)} $$
is real (as $F$ has trivial Nebentype). A direct consequence of Proposition \ref{prop_period_F} and (\ref{eq_Popa_formula}), (\ref{eq_congruence_Heegner_cycle}) and (\ref{eq_bijection_opt_emb_cl})  is the following congruence formula.

\begin{prop}\label{prop_main_congruence_L} 
Let $F \in S_2(\Gamma_0(\ell))$ be a newform satisfying the Eisenstein congruence (\ref{eq_Eis_congruence}). Then for all fundamental discriminant $\Delta$ of a real quadratic field in which $\ell$ splits, we have:
$$\sqrt{\Delta}\cdot \frac{L(F, \chi_{\Delta},1)}{L(F,1)} = L_F(\Delta)^2$$
for some $L_F(\Delta) \in K_F$ (the Hecke field of $F$), uniquely defined up to sign, such that 
\begin{equation}\label{eq_trivial_divisibility}
L_F(\Delta) \in I_F\cdot \mathcal{O}_{\mathfrak{P}_F}
\end{equation}
and the image $\widetilde{L_F(\Delta)}$ of $L_F(\Delta)$ in $\Z/p\Z$ via the map (\ref{eq_I_F/I_F^2}) satisfies
$$
\widetilde{L_F(\Delta)} =  \pm h_{\Delta} \cdot \log_{\l, p}(\epsilon_{\Delta})
$$
for any prime $\mathfrak{l}$ above $\ell$ in $K=\Q(\sqrt{\Delta})$. 
\end{prop}

\section{Half-integral weight modular forms}\label{section_half_int_wt}

\subsection*{Notation}
We keep the same notation as in section \ref{section_Eisenstein_congruences}. In particular, we still assume $p\geq 5$. In all this section, we let
$$\chi : (\Z/\ell\Z)^{\times}\rightarrow 
\C^{\times}$$ be any \emph{odd} Dirichlet character, \ie such that $\chi(-1)=-1$. We let $\chi'$ be the unique \emph{even} Dirichlet character of level $4\ell^2$ whose restriction to $(\Z/\ell^2\Z)^{\times}$ is induced by $\chi$. Explicitly, $\chi' : (\Z/4\ell^2\Z)^{\times}\rightarrow \C^{\times}$ is given by
\begin{equation}\label{eq_def_chi}
\chi'(x) = (-1)^{\frac{x-1}{2}}\cdot \chi(x) \text{.}
\end{equation}

Let $$R = \Z[\chi]=\Z[\chi']$$ be the subring of $\C$ in which $\chi$ takes its values. Similarly, we let $\Q(\chi)\subset \C$ be the fraction field of $R$. 

Note that if $\ell \equiv 3 \text{ (modulo } 4\text{)}$ (\eg if $\ell=11$, as will be the case in \S \ref{section_cor_BSD}), one may choose for instance $\chi(x) = \left(\dfrac{x}{\ell}\right)$, in which case $R = \Z$, which would simplify the notation in our arguments.

\subsection{The generalized Shimura lifting of Baruch--Mao}
The goal of this section is to combine the congruence of Proposition \ref{prop_main_congruence_L} with a generalization of the Kohnen--Zagier formula for half-integral weights modular forms due to \cite{Mao} in order to construct a modulo $p$ weight $\frac{3}{2}$ modular form whose $\Delta$th coefficient controls the vanishing of $h(\Delta) \cdot \log_{\l, p}(\epsilon_{\Delta})$. Our main result is stated in Theorem \ref{prop_main_wt_3/2} below. 

There is a well-defined notion of half-integral weight modular form with coefficients in any (commutative) ring $A$ in which the level is invertible, \cf \cite{Ramsey}. Thus, if $M \in \mathbf{N}$ and $A$ is a commutative ring with $4M \in A^{\times}$, one can consider the space $S_{\frac{3}{2}}(\Gamma_1(4M), \psi)_{A}$ of weight $\frac{3}{2}$ cuspforms with coefficients in $A$ of level $\Gamma_1(4M)$ and character $\psi : (\Z/4M\Z)^{\times} \rightarrow A^{\times}$. If we do not indicate the base ring $A$ in the notation, it means we take $A=\C$.

Let $F \in S_2(\Gamma_0(\ell))$ be a newform. Let $w_{\ell}$ be the Atkin--Lehner involution. We assume that 
$$w_{\ell}F=-F \text{,}$$
\ie that the sign of the functional equation of $L(F,s)$ is $1$. 
Recall that, as $\ell$ is prime, we have $w_{\ell} = -U_{\ell}$ on $S_2(\Gamma_0(\ell))$. Therefore $U_{\ell}$ is an involution, and our assumption amounts to $U_{\ell}F=F$. 

\begin{rem}
This assumption holds in particular if $F$ satisfies the Eisenstein congruence (\ref{eq_Eis_congruence}). Indeed, in this case we have $U_{\ell}F \equiv F \text{ (modulo }\mathfrak{P}_F \text{)}$ and, since $p>2$, we conclude that $U_{\ell}F = F$ (this is well-known, even for $p=2$, and follows from \cite[Proposition II.17.10]{Mazur_Eisenstein}) As recalled before, for such a residually Eisenstein $F$, we actually have the stronger fact that $L(F,1) \neq 0$.
\end{rem}

We now use \cite[Theorem 10.1]{Mao} to construct a weight $\frac{3}{2}$ Hecke eigenform whose Fourier coefficients are related to the twisted $L$-values $L(F, \chi_{\Delta},1)$. The reader may also refer to \cite[Theorem 1.4]{Mao_generalized_Shimura} (which is more general and deals with odd but not necessarily prime levels). 

By applying \cite[Theorem 10.1]{Mao} with $N=N'=\ell$, $S=\{\ell\}$, we get that here exists a non-zero 
$$f_{\chi} = \sum_{n\geq 1} a_n(f_{\chi})q^n \in S_{\frac{3}{2}}(\Gamma_1(4\ell^2), \chi') \text{,}$$ unique up to non-zero scalar, satisfying the following properties:
\begin{enumerate}[label=(\roman*)]
\item\label{Mao_i} $f_{\chi}$ is a Shimura lift of $F$. This condition is a generalization of Shimura's original notion of lift \cite{Shimura_annals}. It is expressed explicitly in a classical language in \cite[\S 1.1]{Mao_generalized_Shimura}. This means that for all prime $q \nmid 2\ell$, we have
$$
T_{q^2}f_{\chi} =  \chi(q)\cdot a_q(F)\cdot f_{\chi}\text{.}
$$
(Recall (\cf \eg \cite[\S 3]{Bruinier}) that 
\begin{equation}\label{eq_def_T_q^2}
a_n(T_{q^2}f_{\chi}) = a_{nq^2}(f_{\chi})+\chi(q)\left(\dfrac{n}{q}\right)a_n(f_{\chi})+q\cdot\chi(q^2)\cdot a_{\frac{n}{q^2}}(f_{\chi}) \text{ .)}
\end{equation}
\item\label{Mao_ii}  The form $f_{\chi}$ is in  the \emph{Kohnen space} (\cf \cite[\S 9.5]{Mao}), \ie if $n \equiv 2,3 \text{ (modulo }4\text{)}$, then $a_n(f_{\chi}) = 0$.
\item\label{Mao_iii}  We have $a_{\Delta}(f_{\chi})=0$ if $\Delta$ is a fundamental discriminant of a real quadratic field in which $\ell$ does not split, \ie $\Delta \not\in \mathcal{D}_{\ell}$. (Here we use the assumption $w_{\ell}F = -F$.)
\end{enumerate}
Furthermore, for all $\Delta \in \mathcal{D}_{\ell} \cup \{1\}$ (\ie such that $\ell$ splits in $\Q(\sqrt{\Delta})$), we have
$$
\abs{a_{\Delta}(f_{\chi})}^2 = C\cdot \sqrt{\Delta}\cdot L(F, \chi_{\Delta},1)
$$
for some constant $C \in \C^{\times}$ independent of $\Delta$ (but depending on $f_{\chi}$). 

From now on, assume furthermore that we have $L(F,1)\neq 0$. We conclude that for all $\Delta \in \mathcal{D}_{\ell}$, we have
$$
\abs{\frac{a_{\Delta}(f_{\chi})}{a_1(f_{\chi})}}^2 = \sqrt{\Delta}\cdot \frac{L(F, \chi_{\Delta},1)}{L(F,1)}\text{.}
$$
We can thus normalize $f_{\chi}$ such that 
$$a_1(f_{\chi})=1 \text{,}$$ in which case we get, for all $\Delta \in \mathcal{D}_{\ell}$:
\begin{equation}\label{eq_key_eq_a_Delta/a_1}
\abs{a_{\Delta}(f_{\chi})}^2 = \sqrt{\Delta}\cdot \frac{L(F, \chi_{\Delta},1)}{L(F,1)}\text{.}
\end{equation}

Let us rewrite this last equality without using absolute values. We denote by $\overline{f_{\chi}} \in S_{\frac{3}{2}}(\Gamma_1(4\ell^2), \overline{\chi'})$ the form obtained by applying the complex conjugation to the Fourier coefficients of $f_{\chi}$. We claim that we have 
\begin{equation}\label{eq_f_chi_bar}
\overline{f_{\chi}} = f_{\chi} \otimes \chi^{-1} \text{,}
\end{equation}
where $f_{\chi} \otimes \chi^{-1}$ is the twist of $f_{\chi}$ by $\chi^{-1}$, which belongs to $S_{\frac{3}{2}}(\Gamma_1(4\ell^2), \overline{\chi'})$ (\cf \cite[Lemma 3.6]{Shimura_annals}). Indeed, by unicity (with the normalization $a_1=1$), it suffices to check that $f_{\chi} \otimes \chi^{-1}$ satisfies the conditions (i), (ii) and (iii) above, which is straightforward using \cite[\S 3 (7)]{Bruinier}.

We thus have 
$$\overline{a_{\Delta}(f_{\chi})} = \chi(\Delta)^{-1}\cdot a_{\Delta}(f_{\chi}) \text{,}$$
so we can rewrite (\ref{eq_key_eq_a_Delta/a_1}) as
\begin{equation}\label{eq_Shimura_lift_f_chi}
(a_{\Delta}(f_{\chi}))^2 = \chi(\Delta)\cdot \sqrt{\Delta}\cdot \frac{L(F, \chi_{\Delta},1)}{L(F,1)}\text{.}
\end{equation}

One can replace $\chi$ (equivalently $\chi'$) by a Galois conjugate $\chi^{\sigma}$ (where $\sigma \in \Gal(\Q(\chi)/\Q)$), and glue all the weight $\frac{3}{2}$ cuspforms obtained this way. This gives rise to a form 
\begin{equation}\label{def_F_chi}
F_{\chi} := (f_{\chi^{\sigma}})_{\sigma}
\end{equation}
with coefficients in $R\otimes_{\Z} \C \simeq \prod_{\sigma} \C$ (where $\sigma$ runs through $\Gal(\Q(\chi)/\Q))$), instead of just $\C$ for the individual forms $f_{\chi^{\sigma}}$. Note that $\chi$ and $\chi'$ can be considered as taking values in $R \otimes_{\Z} \C$ (\ie we identify $\chi(x) \in R$ with $\chi(x) \otimes 1 \in R \otimes_{\Z} \C$, and similarly for $\chi'$), and that $F_{\chi}$ has Nebentype $\chi'$.

We record the conclusion of the above discussion in the following.

\begin{prop}\label{prop_Shimura_lift}
Let $F \in S_2(\Gamma_0(\ell)))$ be a newform such that $L(F,1) \neq 0$ (in particular, we have $w_{\ell}F = -F$).
Let $\chi : (\Z/\ell\Z)^{\times}\rightarrow \C^{\times}$, $\chi' : (\Z/4\ell^2\Z)^{\times} \rightarrow \C^{\times}$ and $R = \Z[\chi] = \Z[\chi']$ be as defined in the notation of this section.

There exists a unique cuspform $F_{\chi} \in S_{\frac{3}{2}}(\Gamma_1(4\ell^2), \chi')_{R\otimes_{\Z} \C}$ satisfying the following properties.
\begin{enumerate}
\item We have $a_1(F_{\chi})=1$.
\item\label{prop_Shimura_lift_ii} For all prime $q \nmid 2\ell$, we have
\begin{equation}\label{eq_T_q^2_f}
T_{q^2}F_{\chi} =  \chi(q)\cdot a_q(F)\cdot F_{\chi}\text{.}
\end{equation}
(Here, as recalled above, $\chi(q)$ is considered in $R \otimes_{\Z}\C$.)
\item If $n \equiv 2,3 \text{ (modulo }4\text{)}$, then $a_n(F_{\chi}) = 0$.
\item We have $a_{\Delta}(F_{\chi})=0$ if $\Delta$ is the discriminant of a real quadratic field in which $\ell$ does not split, \ie $\Delta \not\in \mathcal{D}_{\ell}$.
\item For all $\Delta \in \mathcal{D}_{\ell}$, we have
\begin{equation}\label{eq_Shimura_lift}
(a_{\Delta}(F_{\chi}))^2 = \chi(\Delta)\cdot \sqrt{\Delta}\cdot \frac{L(F, \chi_{\Delta},1)}{L(F,1)}\text{.}
\end{equation}
(Again, $\chi(\Delta)$ is considered in $R \otimes_{\Z}\C$.)
\end{enumerate}
\end{prop}

\begin{rem}
As we shall see in Theorem \ref{prop_main_wt_3/2} (and its proof) below, the restriction $q \nmid 2\ell$ in Proposition \ref{prop_Shimura_lift} (\ref{prop_Shimura_lift_ii}) is not necessary for $q=\ell$, but is necessary for $q=2$, as $F_{\chi}$ is not an eigenvector for the Hecke operator $T_{2^2}$. Recall (\cf \cite[Theorem 1.7]{Shimura_annals}) we have
$$a_n(T_{2^2}(f)) = a_{4n}(f)$$
for all $n\geq 1$ and $f$ cuspform of weight $\frac{3}{2}$. We then see that the condition defining the Kohnen space is not stable by $T_{2^2}$.
\end{rem}

\subsection{An Eisenstein congruence in weight $\frac{3}{2}$}
We now apply the discussion of the previous paragraph to the case where $F$ satisfies the Eisenstein congruence (\ref{eq_Eis_congruence}), to prove the following.

\begin{thm}\label{prop_main_wt_3/2}
Let $F \in S_2(\Gamma_0(\ell)))$ be a newform satisfying the Eisenstein congruence (\ref{eq_Eis_congruence}) (in particular, we know that $L(F,1) \neq 0$). 
Let $\chi : (\Z/\ell\Z)^{\times}\rightarrow \C^{\times}$, $\chi' : (\Z/4\ell^2\Z)^{\times} \rightarrow \C^{\times}$ and $R = \Z[\chi] = \Z[\chi']$ be as defined in the notation of this section. Let $F_{\chi} \in S_{\frac{3}{2}}(\Gamma_1(4\ell^2),\chi')_{R \otimes_{\Z} \C}$ be as in Proposition \ref{prop_Shimura_lift}.

The Fourier coefficients of $F_{\chi}$ lie in $R \otimes_{\Z} K_F$, where $K_F \subset \C$ is the Hecke field of $F$. Furthermore, when considered in $R \otimes_{\Z} K_F\otimes_{\mathcal{O}_F} \mathcal{O}_{\mathfrak{P}_F}$, the Fourier coefficients of $F_{\chi}$ actually lie in the subring $R \otimes_{\Z} \mathcal{O}_F \otimes_{\mathcal{O}_F} \mathcal{O}_{\mathfrak{P}_F}$. Furthermore, the following holds.
\begin{enumerate}[label=(\roman*)]
\item \label{prop_main_wt_3/2_0} We have $$T_{\ell^2}F_{\chi} = 0 \text{,}$$
\ie for all $n\geq 1$, $a_{n\ell^2}(F_{\chi})=0$.
\item\label{prop_main_wt_3/2_i} We have $$F_{\chi} \equiv \theta_{\chi} \text{ (modulo }I_F)$$
where $\theta_{\chi} \in S_{\frac{3}{2}}(4\ell^2, \chi')_R$ is the classical theta series (\cf \cite[\S 2]{Shimura_annals}) given by the Fourier expansion 
$$\theta_{\chi} := \sum_{n\geq 1} n\cdot \chi(n)\cdot q^{n^2} \text{.}$$
In particular, for all $n\geq 1$ which is not a perfect square, we have 
$$a_n(F_{\chi})  \in I_F\cdot (R \otimes_{\Z} \mathcal{O}_F \otimes_{\mathcal{O}_F} \mathcal{O}_{\mathfrak{P}_F}) \text{.}$$
\item\label{prop_main_wt_3/2_ii} For all $n\geq 1$ which is not a perfect square, we denote by $\widetilde{a_n(F_{\chi})}$ the image of $a_n(F_{\chi})$ in $R/pR$ via the map (\ref{eq_I_F/I_F^2}). For all $\Delta \in \mathcal{D}_{\ell}$, we have in $R/pR$:
\begin{equation}\label{eq_a_Delta_mod_I^2}
\widetilde{a_{\Delta}(F_{\chi})} =  \pm\chi(\sqrt{\Delta} \text{ modulo }\mathfrak{l})\cdot h(\Delta)\cdot \log_{\l, p}(\epsilon_{\Delta}) \text{,}
\end{equation}
for any prime $\mathfrak{l}$ above $\ell$ in $K=\Q(\sqrt{\Delta})$ and for some sign $\pm$ depending \emph{a priori} on $\Delta$. (The right-hand side does not depend on the choice of $\mathfrak{l}$.)
\item\label{prop_main_wt_3/2_iii}
Let $n\geq 1$ which is not a perfect square, and denote by $\Delta$ be the discriminant of the real quadratic field $\Q(\sqrt{n})$. If $\widetilde{a_n(F_{\chi})} \neq 0$ (in $R/pR$), then $\widetilde{a_{\Delta}(F_{\chi})} \neq 0$ (in particular $\Delta \in \mathcal{D}_{\ell}$, \ie $\ell$ splits in $\Q(\sqrt{n})$).
\end{enumerate}
\end{thm}

\begin{rem}\label{rem_theta_series}
\begin{enumerate}
\item\label{rem_theta_series_1} In the special case $\ell=11$ and $p=5$, the congruence of Theorem \ref{prop_main_wt_3/2} \ref{prop_main_wt_3/2_i} follows from the explicit computations of \cite[\S 10.7]{Mao} (note that there is a typo in the last paragraph of p. 381: the $n$th Fourier coefficients of the theta series should be $n\cdot \left(\dfrac{n}{11}\right)$ and not $\left(\dfrac{n}{11}\right)$).
\item\label{rem_theta_series_2} We check that $\theta_{\chi}$ is an eigenform for the Hecke operator $T_{q^2}$ for \emph{all} primes $q$, with Hecke eigenvalue $(q+1)\cdot \chi(q)$ if $q \nmid 2\ell$, $0$ if $q=\ell$ and $2\chi(2)$ if $q=2$.
\item\label{rem_theta_series_3} Although \cite[Theorem 10.1]{Mao} only applies \emph{a priori} when $F$ is a cuspform, we see that, formally speaking, the weight $\frac{3}{2}$-modular form associated to the weight $2$ Eisenstein series $E_{2,\ell} \in M_2(\Gamma_0(\ell))$ is simply $\theta_{\chi} \in S_{\frac{3}{2}}(\Gamma_0(4\ell^2), \chi')$. Note also that (\ref{eq_Shimura_lift}) also holds, as $L(E_{2,\ell}, \chi_D, 1)=0$ if $D \in \mathcal{D}_{\ell}$. Therefore, Theorem \ref{prop_main_wt_3/2} \ref{prop_main_wt_3/2_i} reflects the fact that the Eisenstein congruence (\ref{eq_Eis_congruence}) is in some sense compatible with the (generalized) Shimura lifting. 
\item\label{rem_theta_series_4} One can ask what is the sign in (\ref{eq_a_Delta_mod_I^2}), and in particular whether it can be taken to be equal to $1$.
\end{enumerate}
\end{rem}
\begin{proof}
This essentially follows from (\ref{eq_Shimura_lift}) and Proposition \ref{prop_main_congruence_L}, except that some difficulties arise due to the restriction $q \nmid 2\ell$ in (\ref{eq_T_q^2_f}). The main issue is that $f_{\chi}$ is \emph{not} an eigenvector for the Hecke operator $T_{2^2}$. 
However, as we shall see, $f_{\chi}$ is a linear combination of two cusp forms in $S_{\frac{3}{2}}(\Gamma_1(4\ell^2), \chi')$ which are eigenforms for \emph{all} the Hecke operators $T_{q^2}$ and satisfy similar properties as the ones $f_{\chi}$ described in Proposition \ref{prop_Shimura_lift} (except for the third one). We will need to use the adelic language, following and relying on \cite{Mao} and \cite{Waldspurger_demi_entier}. 

Proof of \ref{prop_main_wt_3/2_0} and \ref{prop_main_wt_3/2_i}.  As in \cite[\S 9.1]{Mao}, corresponding to $F$ is a vector
$$\varphi = \otimes' \varphi_v  = s(F) \in \pi \text{,}$$
where 
$$\pi = \otimes_v'\pi_v$$
is the irreducible cuspidal automorphic representation of $\PGL_2$ corresponding to $F$.

As in \cite[\S 9.2]{Mao}, corresponding to $f_{\chi}$ is a vector 
$$\tilde{\phi}_{\chi} = \otimes' \tilde{\varphi}_v = t(f_{\chi}) \in \tilde{\pi} \text{,}$$
where 
$$\tilde{\pi} = \otimes'_v\tilde{\pi}_v$$
is an irreducible cuspidal automorphic representation of the metaplectic group given by
$$\tilde{\pi} = \Theta(\pi \otimes \chi_D, \psi^D)$$
for any $D \in \mathcal{D}_{\ell}$ (\ie such that $D>0$ is a non-zero square modulo $\ell$).
Here, $\chi_D$ is the quadratic Dirichlet character corresponding to $D$, $\psi$ is the usual additive character and $\Theta$ is the theta correspondence (we refer to \cite[\S 3]{Mao} for details).

If $q$ is a prime, let $\chi_q$ be the character of $\Q_q^{\times}$ which is the component at $q$ of the character on $\mathbb{A}_{\Q}^{\times}/\Q^{\times}$ canonically associated with $\chi$. In particular, for $q \neq \ell$, the character $\chi_q$ is unramified and we have $\chi_q(q) = \chi(q)$.

For $q \mid 2\ell$, Waldspurger defines in \cite{Waldspurger_demi_entier} a Hecke operator $\tilde{T}_q'$ acting on $\tilde{\phi}_{\chi}$ and we have (\cf \cite[III.B Lemma 4 (ii)]{Waldspurger_demi_entier}) that we have
\begin{equation}\label{eq_compatibility_Hecke}
T_{q^2}f_{\chi} = \sqrt{q}^{-1}\cdot\tilde{\gamma}_q(q)^{-1}\cdot \chi_q(q)\cdot s(\tilde{T}_q' \tilde{\phi}_{\chi}) \text{,}
\end{equation}
where $\tilde{\gamma}_q(q)$ is the Weil constant (\cf \cite[II \S 2]{Waldspurger_demi_entier}). As recalled in \cf \cite[II \S 6]{Waldspurger_demi_entier}, we have $\tilde{\gamma}_2(2)=1$.

In \cite[\S 8]{Mao}, an explicit description of $\tilde{\pi}_v$ and $\tilde{\varphi}_v$ (up to scalar) is given for all places $v \neq 2$. In particular, if $v=\ell$, from our condition $D \in \mathcal{D}_{\ell}$ and the fact that $w_{\ell}F=-F$, we get by \cite[\S 10.5]{Mao} that $\tilde{\pi}_{\ell}$ is the supercuspidal representation $r_{\psi}^-$ described in \cite[\S 8.3.3]{Mao} and $\tilde{\varphi}_{\ell}$ is, up to scalar, an explicit function on $\Q_{\ell}$ in $r_{\psi}^-$ denoted by $\Phi_{\chi_{\ell}}$ in \cite[Proposition 8.5]{Mao}. Note that for all $x \in \Z_{\ell}^{\times}$, we have
$$\chi_{\ell}(x) = \overline{\chi}(x \text{ (modulo }\ell\text{)}) \text{.}$$

By definition, the function $\Phi_{\chi_{\ell}}$ is supported on $\Z_{\ell}^{\times}$. By \cite[V \S 2 Lemme 7 (iii)]{Waldspurger_demi_entier} (with $n=2$), we see that $\tilde{T}_{\ell}'\tilde{\varphi}_{\ell} = 0$. By (\ref{eq_compatibility_Hecke}), we get:
\begin{equation}\label{eq_T_ell^2_phi}
T_{\ell^2}f_{\chi} = 0 \text{.}
\end{equation}
We now consider the place $v=2$. The choice of $\tilde{\varphi}_2$ is given in details in \cite[\S 9.4]{Mao}. Since $a_2(F) \equiv 3 \text{ (modulo }\mathfrak{P}_F \text{)}$ and the polynomial $X^2-3X+2 = (X-1)(X-2)$ splits over $\mathbf{F}_p$, by Hensel's lemma we get 
$$X^2-a_2(F)X+2 = (X-\alpha_1)(X-\alpha_2)$$
for some $\alpha_1, \alpha_2\in \mathcal{O}_{\mathfrak{P}_F}$, which we also view in $\C$. Without loss of generality, let us assume that 
\begin{equation}\label{eq_congruence_alpha_1}
\alpha_1 \equiv 1 \text{ (modulo }I_F \text{).}
\end{equation}
For simplicity of notation, we let $\alpha := \alpha_1$ in what follows. Note that 
$$\alpha \neq 1 \text{,}$$ 
as otherwise we would have $a_2(F)=3$, which would contradict the Ramanujan bound $\abs{a_2(F)} \leq 2\sqrt{2}$.

Let $\mu : \Q_2^{\times} \rightarrow \C^{\times}$ be the unramified character of $\Q_2^{\times}$ such that 
$$\mu(2) = \frac{\alpha}{\sqrt{2}} \text{.}$$
Note that we have $\overline{\alpha_1} = \alpha_2$, so 
$$\overline{\mu(2)} = \frac{\alpha_2}{\sqrt{2}} = \frac{\sqrt{2}}{\alpha_1} = \mu(2)^{-1}$$
\ie $\mu(2)$ is a root of unity.
By construction, we have
\begin{equation}\label{eq_congruence_mu_2^2}
\mu(2^2) \equiv 2^{-1} \text{ (modulo }I_F \text{).} 
\end{equation}
In particular, we get $\mu(2^2) \neq 1$.

In \cite[VI \S 8 Proposition 13 (iii)]{Waldspurger_demi_entier} (with $\chi_0$ trivial and $n=2$), Waldspurger proves that there are two vectors $F[2,1]$ and $F[2,2^2]$ of $\tilde{\pi}_2$ such that
$$\tilde{T}_2'F[2,1] = 2\mu(2)^{-1}\cdot F[2,1] $$
and
$$\tilde{T}_2'F[2,2^2] = 2\mu(2)\cdot F[2,2^2] + \frac{i+1}{2}\cdot \mu(2)\cdot F[2,1] \text{.}$$
Therefore, the two vectors $$e_1 := F[2,1]$$ and $$e_2 := \frac{i+1}{4}\cdot\frac{\mu(2^2)}{\mu(2^2)-1}\cdot F[2,1]+F[2,2^2]$$ are eigenvectors for $\tilde{T}_2'$ with (distinct) eigenvalues $2\mu(2)^{-1}$ and $2\mu(2)$ respectively (note that we use the fact that $\mu(2^2) \neq 1$).

Consider now the two weight $\frac{3}{2}$ cuspforms of level $\Gamma_1(4\ell^2)$ and character $\chi'$
$$g_{1, \chi} := s(\otimes_{v\neq 2} \tilde{\varphi}_v \otimes_{v=2} e_1)$$
and
$$g_{2, \chi} := s(\otimes_{v\neq 2} \tilde{\varphi}_v \otimes_{v=2} e_2) \text{.}$$
These are, by construction, eigenforms for all Hecke operators $T_{q^2}$, such that, for $i \in \{1,2\}$, we have
$$T_{\ell^2}g_{i,\chi}  = 0 \text{,}$$
$$T_{2^2}g_{i, \chi} = 2\chi(2)\alpha_i^{-1}\cdot g_{i, \chi} \text{,}$$
and for all primes $q \nmid 2\ell$,
$$T_{q^2}g_{i, \chi} = \chi(q)\cdot a_q(F)\cdot g_{i, \chi} \text{.}$$

By \cite[Theorem 4.3]{Mao}, we have $a_1(g_{i,\chi}) \neq 0$ for $i=1,2$. Thus, there exists $\lambda_{i,\chi} \in \C^{\times}$ such that $f_{i,\chi} := \lambda_{i,\chi}\cdot g_{i,\chi}$ satisfies $a_1(f_{i,\chi})=1$. Obviously, $f_{i,\chi}$ satisfies the same Hecke relations as $g_{i,\chi}$, \ie
\begin{equation}\label{eq_Hecke_relation_f_i_chi}
\begin{array}{lcl}
		T_{\ell^2}f_{i,\chi}  &=& 0, \\
		 T_{2^2}f_{i,\chi}  &=& 2\chi(2)\alpha_i^{-1}\cdot f_{i, \chi},\\
		\forall q\nmid 2\ell \text{ (}q\text{ prime) } T_{q^2}f_{i, \chi} &=& \chi(q)\cdot a_q(F)\cdot f_{i, \chi}  \text{.}
	\end{array}
\end{equation}
These Hecke relations, together with the normalization $a_1(f_{i,\chi})=1$, characterize uniquely $f_{i,\chi}$. In particular, using the relation $\overline{\alpha_1} = \alpha_2$, a similar argument as for (\ref{eq_f_chi_bar}) yields
\begin{equation}\label{eq_f_chi_bar_1_2}
\overline{f_{1,\chi}} = f_{2, \chi^{-1}} = f_{2,\chi} \otimes \chi^{-1} \text{.}
\end{equation}

Let us now compute the ratio $\frac{\lambda_{1,\chi}}{\lambda_{2,\chi}} \in \C^{\times}$. By \cite[\S 2.2]{Mao}, we have
$$\frac{a_1(g_{1,\chi})}{a_1(g_{2, \chi})} = \frac{\tilde{L}_2(e_1)}{\tilde{L}_2(e_2)} \text{,}$$
where $\tilde{L}_2$ is any choice (unique up to non-zero scalar) of local Whittaker functional at the place $2$ for $\tilde{\pi}_2$. By \cite[Lemma 9.3]{Mao}, one can normalize $\tilde{L}_2$ such that
$$\tilde{L}_2(F[2,1])=1$$
and
$$\tilde{L}_2(F[2,2^2])=(\mu(2^2)+\sqrt{2}\mu(2^3))\cdot\frac{i+1}{4} \text{.}$$
We choose this normalization for the rest of the proof.

We get
\begin{equation}\label{eq_ratio_lambdas}
\frac{\lambda_{1,\chi}}{\lambda_{2,\chi}} = \frac{a_1(g_{2,\chi})}{a_1(g_{1,\chi})} =\frac{i+1}{4}\cdot\frac{\mu(2^2)}{\mu(2^2)-1} + \frac{i+1}{4}\cdot (\mu(2^2)+\sqrt{2}\mu(2^3))  =\frac{i+1}{4}\cdot\frac{\mu(2^2)}{\mu(2^2)-1} + \frac{i+1}{4}\cdot (\mu(2^2)+\alpha\mu(2^2))  \text{.}
\end{equation}
(We have used the definition $\mu(2) = \frac{\alpha}{\sqrt{2}}$ in the last equality.)

By \cite[\S 9.4]{Mao}, it is proved that $\tilde{\varphi}_2$ is proportional to
$$
\frac{i+1}{4}\cdot \mu(2^2)\cdot F[2,1]+F[2,2^2] = \frac{i+1}{4}\cdot \frac{\mu(2^4)}{\mu(2^2)-1}\cdot e_1 + e_2  \text{.}
$$
Thus, $f_{\chi}$ is proportional to $\frac{i+1}{4}\cdot \frac{\mu(2^4)}{\mu(2^2)-1}\cdot g_{1, \chi} + g_{2,\chi}$, which is proportional to 
$$
\frac{i+1}{4}\cdot \frac{\mu(2^4)}{\mu(2^2)-1}\cdot f_{1,\chi} + \frac{\lambda_{1,\chi}}{\lambda_{2,\chi}}\cdot f_{2,\chi} \text{.}
$$
By (\ref{eq_ratio_lambdas}), there exists $\lambda \in \C^{\times}$ such that
$$
f_{\chi} = \lambda\cdot \left( \frac{\mu(2^4)}{\mu(2^2)-1}\cdot f_{1, \chi} + \left(\frac{\mu(2^2)}{\mu(2^2)-1} +  \mu(2^2)+\alpha_1\mu(2^2) \right)\cdot f_{2, \chi} \right) \text{.}
$$
Since $a_1(f_{\chi}) = a_1(f_{1, \chi}) = a_1(f_{2, \chi})=1$, we get
$$
\lambda = \frac{\mu(2^2)-1}{2\mu(2^4)-\alpha_1\mu(2^2)+\alpha_1\mu(2^4)} \text{.}
$$
We get
\begin{equation}\label{eq_decomposition_fchi}
f_{\chi} = \mu_1 f_{1, \chi}+\mu_2f_{2,\chi}
\end{equation}
where $\mu_1,\mu_2 \in K_F(\alpha)$ (which is at most a quadratic extension of $K_F$) are such that
\begin{equation}\label{eq_mu_1}
\mu_1 = \frac{\alpha}{\alpha^2+2\alpha-2}
\end{equation}
and
\begin{equation}\label{eq_mu_2}
\mu_2 = \frac{(\alpha-1)(\alpha+2)}{\alpha^2+2\alpha-2} \text{.}
\end{equation} 
By (\ref{eq_congruence_alpha_1}), we get
\begin{equation}\label{eq_congruence_mu_1}
\mu_1 \equiv 1 \text{ (modulo } I_F \text{)}
\end{equation}
and
\begin{equation}\label{eq_congruence_mu_2}
\mu_2 \equiv 0 \text{ (modulo } I_F \text{).}
\end{equation}

Let us now study the coefficients $a_D(f_{i,\chi})$ where $D$ ($\neq 1$) is the discriminant of a real quadratic field.

Assume first $D \not\in \mathcal{D}_{\ell}$. In this case, it follows from \cite[Theorem 4.3 2. and Corollary 10.7]{Mao} that
\begin{equation}\label{eq_vanishing_a_D}
a_D(f_{\chi}) = a_D(f_{1,\chi}) = a_D(f_{2,\chi}) = 0 \text{.}
\end{equation}

Assume now $D \in \mathcal{D}_{\ell}$. It follows from \cite[Theorem 4.3 3.]{Mao} that $a_D(f_{\chi})=0$ if and only if $L(F, \chi_D, 1)=0$, if and only if $a_D(f_{i,\chi})=0$ (for $i=1,2$). We thus assume that $a_D(f_{\chi}) \neq 0$. We first compute the (well-defined) ratio $\frac{a_D(f_{\chi})}{a_D(f_{1,\chi})}$. It follows easily from \cite[\S 2.2]{Mao} that we have
$$
\frac{a_D(f_{\chi})}{a_D(f_{1,\chi})} = \frac{a_1(f_{\chi})}{a_1(f_{1,\chi})}\cdot \left( \frac{\tilde{L}_2(\frac{i+1}{4}\cdot \frac{\mu(2^4)}{\mu(2^2)-1}\cdot e_1 + e_2 )}{\tilde{L}_2(e_1)}\right)^{-1}\cdot \frac{\tilde{L}_2^D(\frac{i+1}{4}\cdot \frac{\mu(2^4)}{\mu(2^2)-1}\cdot e_1 + e_2 )}{\tilde{L}_2^D(e_1)} \text{.}
$$
Here, as in \cite[\S 2.2]{Mao}, $\tilde{L}_2^D$ denotes a local Whittaker functional at the place $v=2$ corresponding to an additive character $\psi^D$, which we normalize as in \cite[\S 8.1 (8.3)]{Mao}. Recall that, by construction, we have $\frac{a_1(f_{\chi})}{a_1(f_{1,\chi})}=1$.

By \cite[Lemma 9.3]{Mao}, we have 
$$\tilde{L}_2(e_1) = \tilde{L}_2^D(e_1) = 1 \text{,}$$
and 
$$
\tilde{L}_2(\frac{i+1}{4}\cdot \frac{\mu(2^4)}{\mu(2^2)-1}\cdot e_1 + e_2 ) = \frac{i+1}{4}\cdot \frac{\mu(2^4)}{\mu(2^2)-1} + \frac{i+1}{4}\cdot (\mu(2^2)+\alpha\cdot \mu(2^2)) \text{.}
$$

We also get
$$
\tilde{L}_2^D(\frac{i+1}{4}\cdot \frac{\mu(2^4)}{\mu(2^2)-1}\cdot e_1 + e_2 ) = \frac{i+1}{4}\cdot \frac{\mu(2^4)}{\mu(2^2)-1} + \frac{i+1}{4}\cdot (\mu(2^2)+\alpha\cdot \mu(2^2)\cdot \chi_2(D)) \text{ if } D \equiv 1 \text{ (modulo }4\text{)} 
$$
and
$$
\tilde{L}_2^D(\frac{i+1}{4}\cdot \frac{\mu(2^4)}{\mu(2^2)-1}\cdot e_1 + e_2 ) = \frac{i+1}{4}\cdot \frac{\mu(2^4)}{\mu(2^2)-1} + \frac{i+1}{4}\cdot (\mu(2^2)-\mu(2^4)) \text{ if } D \equiv 0 \text{ (modulo }4\text{).} 
$$
Here, $\chi_2$ is the quadratic character (of conductor $8$) of $\Q_2^{\times}$ corresponding to the extension $\Q_2(\sqrt{2})$.

In conclusion, we get
$$
\frac{a_D(f_{\chi})}{a_D(f_{1,\chi})} = \frac{\frac{\mu(2^4)}{\mu(2^2)-1} + \mu(2^2)+\alpha\cdot \mu(2^2)\cdot \chi_2(D)}{\frac{\mu(2^4)}{\mu(2^2)-1} + \mu(2^2)+\alpha\cdot \mu(2^2)} \text{ if } D \equiv 1 \text{ (modulo }4\text{)} 
$$
and
$$
\frac{a_D(f_{\chi})}{a_D(f_{1,\chi})} = \frac{ \frac{\mu(2^4)}{\mu(2^2)-1} +  \mu(2^2)-\mu(2^4)}{\frac{\mu(2^4)}{\mu(2^2)-1} + \mu(2^2)+\alpha\cdot \mu(2^2)} \text{ if } D \equiv 0\text{ (modulo }4\text{).} 
$$

Since, for $D \equiv 1 \text{ (modulo }4\text{})$, we have $\chi_2(D)=1$ if $D \equiv 1 \text{ (modulo }8\text{})$ and $\chi_2(D)=-1$ if $D \equiv 5 \text{ (modulo }4\text{})$, this simplifies to give

\begin{equation}\label{eq_ratio_a_D_1_1}
a_D(f_{1,\chi}) = a_D(f_{\chi}) \text{ if } D \equiv 1 \text{ (modulo }8\text{),} 
\end{equation}

\begin{equation}\label{eq_ratio_a_D_5_1}
a_D(f_{1,\chi}) = -\frac{\alpha^3+2\alpha^2-2\alpha-2}{\alpha^3-2\alpha^2-2\alpha+2} \cdot a_D(f_{\chi}) \text{ if } D \equiv 5 \text{ (modulo }8\text{),} 
\end{equation}
and
\begin{equation}\label{eq_ratio_a_D_0_1}
a_D(f_{1,\chi}) = -2\cdot \frac{\alpha^3+2\alpha^2-2\alpha-2}{\alpha^4-6\alpha^2+4} \cdot a_D(f_{\chi})  \text{ if } D \equiv 0 \text{ (modulo }4\text{).} 
\end{equation}

By (\ref{eq_decomposition_fchi}) (and the fact that $\alpha \neq 1$, as recalled above), we get
\begin{equation}\label{eq_ratio_a_D_1_2}
a_D(f_{2,\chi}) = a_D(f_{\chi})  \text{ if } D \equiv 1 \text{ (modulo }8\text{) , } 
\end{equation}
\begin{equation}\label{eq_ratio_a_D_5_2}
a_D(f_{2,\chi}) =  \frac{\alpha^5 + \alpha^4 - 6\alpha^3 + 6\alpha - 4}{(\alpha - 1)(\alpha + 2)(\alpha^3 - 2\alpha^2 - 2\alpha + 2)}\cdot a_D(f_{\chi}) \text{ if } D \equiv 5\text{ (modulo }8\text{).} 
\end{equation}
and
\begin{equation}\label{eq_ratio_a_D_0_2}
a_D(f_{2,\chi}) = \frac{\alpha^5 - 6\alpha^3 + 4\alpha^2 + 4\alpha - 4}{(\alpha-1)(\alpha^4 - 6\alpha^2 + 4)} \cdot a_D(f_{\chi})  \text{ if } D \equiv 0 \text{ (modulo }4\text{).} 
\end{equation}

Note that the six labeled equations above also hold if $a_D(f_{\chi})=0$, as both sides are zero as recalled above.

Combining (\ref{eq_Shimura_lift}), Proposition \ref{prop_main_congruence_L}, (\ref{eq_ratio_a_D_1_1}), (\ref{eq_ratio_a_D_5_1}), (\ref{eq_ratio_a_D_0_1}),  (\ref{eq_ratio_a_D_5_2}) and (\ref{eq_ratio_a_D_0_2}), we conclude for all $D \in \mathcal{D}_{\ell}$ and $i \in \{1,2\}$, the Fourier coefficient $\frac{a_D(f_{i, \chi})}{{\chi(\sqrt{D} \text{ modulo }\mathfrak{l})}}$ is algebraic (where $\mathfrak{l}$ is any prime above $\ell$ in $\Q(\sqrt{D})$), and in fact belongs to the (at most) quadratic extension $K_F(\alpha)$ of $K_F$ in $\C$. Furthermore, using (\ref{eq_trivial_divisibility}),(\ref{eq_congruence_alpha_1}) and (\ref{eq_mu_2}), for all $D \in \mathcal{D}_{\ell}$, we get in $K_F(\alpha) \otimes_{\mathcal{O}_F} \mathcal{O}_{\mathfrak{P}_F}$:
\begin{equation}\label{eq_three_a_D_I_F}
\frac{a_D(f_{\chi})}{\chi(\sqrt{D} \text{ modulo }\mathfrak{l})}, \frac{a_D(f_{1, \chi})}{\chi(\sqrt{D} \text{ modulo }\mathfrak{l})},  \mu_2\cdot \frac{a_D(f_{2, \chi})}{\chi(\sqrt{D} \text{ modulo }\mathfrak{l})} \in I_F \cdot \mathcal{O}_{\mathfrak{P}_F} \text{.}
\end{equation}
Here, we view $\mathcal{O}_{\mathfrak{P}_F}$ as a subring of $K_F(\alpha) \otimes_{\mathcal{O}_F} \mathcal{O}_{\mathfrak{P}_F}$ in the obvious way. Let us note that (\ref{eq_three_a_D_I_F}) also hold for fundamental discriminants $D \not\in \mathcal{D}_{\ell}$, as by (\ref{eq_vanishing_a_D}) these three quantities equal zero.

Recall that we have defined in (\ref{def_F_chi}) the form $F_{\chi} \in S_{\frac{3}{2}}(\Gamma_1(4\ell^2), \chi')_{R\otimes_{\Z}\C}$ by gluing the forms $f_{\chi^{\sigma}}$ where $\sigma$ runs through $\Gal(\Q(\chi)/\Q)$. We define similarly $F_{i, \chi} \in S_{\frac{3}{2}}(\Gamma_1(4\ell^2), \chi')_{R\otimes_{\Z}\C}$ for $i=1,2$ by gluing $f_{i,\chi^{\sigma}}$. The equations (\ref{eq_Hecke_relation_f_i_chi}) and (\ref{eq_decomposition_fchi}) with $f_{\chi}$, $f_{1, \chi}$ and $f_{2, \chi}$ replaced by $F_{\chi}$, $F_{1, \chi}$ and $F_{2, \chi}$ respectively obviously hold. 

By \cite[\S 2.2]{Mao} and the choice of local Whittaker functional at $v=\ell$ in \cite[p. 367]{Mao}, we get for $\sigma, \sigma' \in \Gal(\Q(\chi)/\Q)$ and $D \in \mathcal{D}_{\ell}$:
\begin{equation}\label{eq_comparison_chi_chi'}
\begin{array}{lcl}
		a_D(f_{\chi^{\sigma}})  &=& (\frac{\chi^{\sigma}}{\chi^{\sigma'}})(\sqrt{D} \text{ modulo } \mathfrak{l})\cdot a_D(f_{\chi^{\sigma'}}), \\
		a_D(f_{i, \chi^{\sigma}})  &=& (\frac{\chi^{\sigma}}{\chi^{\sigma'}})(\sqrt{D} \text{ modulo } \mathfrak{l})\cdot a_D(f_{i,\chi^{\sigma'}}) \text{ (for }i\in \{1,2\}\text{).}
	\end{array}
\end{equation}
Note that these equations do not depend on the choice of $\mathfrak{l}$, as $\frac{\chi^{\sigma'}}{\chi^{\sigma}}$ is an even character of level $\ell$. It follows from the discussion above (\ref{eq_three_a_D_I_F}) that for all fundamental discriminant $D$ (with $D\neq 1$), we have 
$$a_D(F_{\chi}), a_D(F_{1,\chi}), a_D(F_{2,\chi}) \in R \otimes_{\Z} K_F(\alpha) \text{,}$$
and furthermore, using (\ref{eq_three_a_D_I_F}) and (\ref{eq_comparison_chi_chi'}), we conclude that we actually have

\begin{equation}\label{eq_three_a_D_I_F_final}
a_D(F_{\chi}), a_D(F_{1, \chi}),  \mu_2\cdot a_D(F_{2, \chi}) \in I_F \cdot (R\otimes_{\Z}\mathcal{O}_{\mathfrak{P}_F}) \text{,}
\end{equation}
where $R\otimes_{\Z}\mathcal{O}_{\mathfrak{P}_F}$ is considered as a subring of $R\otimes_{\Z}K_F(\alpha) \otimes_{\mathcal{O}_F}\mathcal{O}_{\mathfrak{P}_F}$.

Using (\ref{eq_Hecke_relation_f_i_chi}), (\ref{eq_decomposition_fchi}), (\ref{eq_congruence_mu_1}) and (\ref{eq_congruence_mu_2}), we conclude that 
for all positive integer $n$, we have
$$a_n(F_{\chi}) \in R \otimes_{\Z} K_F$$
and that furthremore, if $n$ is not a perfect square,
$$a_n(F_{\chi}) \in I_F\cdot( R \otimes_{\Z}\mathcal{O}_{\mathfrak{P}_F})$$
whereas if $n=m^2$ is a perfect square, 
$$a_n(F_{\chi}) - m\cdot \chi(m) \in  I_F\cdot( R \otimes_{\Z}\mathcal{O}_{\mathfrak{P}_F}) \text{.}$$
This concludes the proof of \ref{prop_main_wt_3/2_i}. 

Point \ref{prop_main_wt_3/2_ii} follows immediately from (\ref{eq_Shimura_lift}) and Proposition \ref{prop_main_congruence_L}.

Proof of \ref{prop_main_wt_3/2_iii}. We have $n = m^2\cdot \Delta$ where $2m \in \mathbf{N}$. If $m$ is an odd integer, then by (\ref{eq_T_q^2_f}) and (\ref{eq_def_T_q^2}), we see that $a_n(F_{\chi})$ is an explicit multiple of $a_{\Delta}(F_{\chi})$, which proves the claim in point \ref{prop_main_wt_3/2_iii}. 

Assume now that $m$ is an even integer. By the same argument as when $m$ is odd, we may assume without loss of generality that $m=2^a$ for some $a \geq 1$. Combining (\ref{eq_Hecke_relation_f_i_chi}), (\ref{eq_decomposition_fchi}), (\ref{eq_ratio_a_D_1_1}), (\ref{eq_ratio_a_D_5_1}), (\ref{eq_ratio_a_D_0_1}), (\ref{eq_ratio_a_D_1_2}), (\ref{eq_ratio_a_D_5_2}), and (\ref{eq_ratio_a_D_0_2}) (depending on $\Delta$ modulo $8$), we see again that $a_n(F_{\chi})$ is an explicit multiple of $a_{\Delta}(F_{\chi})$, which proves the claim in point \ref{prop_main_wt_3/2_iii}. 

Finally, assume that $m$ is half an odd integer. Again, without loss of generality, we may assume $m=\frac{1}{2}$, \ie $\Delta = 4n$. By (\ref{eq_Hecke_relation_f_i_chi}), (\ref{eq_decomposition_fchi}), (\ref{eq_ratio_a_D_0_1}), and (\ref{eq_ratio_a_D_0_2}), we get
\begin{align*}
a_{\Delta}(F_{\chi}) &= a_n(T_{2^2}F_{\chi}) 
\\& = \mu_1\cdot 2\chi(2)\alpha^{-1}\cdot a_n(F_{1, \chi})+ \mu_2\cdot 2\chi(2)\cdot (\frac{2}{\alpha})^{-1}\cdot a_n(F_{2, \chi})
\\& = \chi(2)\alpha^{-1}\cdot (2\mu_1\cdot a_n(F_{1, \chi})+\mu_2\cdot \alpha^2\cdot a_n(F_{2,\chi}))
\\& = 2\chi(2)\alpha^{-1}\cdot a_n(F_{\chi}) + \mu_2\cdot (\alpha^2-2)\cdot a_n(F_{2,\chi})
\\& = 2\chi(2)\alpha^{-1}\cdot a_n(F_{\chi}) + \mu_2\cdot (\alpha^2-2)\cdot \chi(2)^{-1}\alpha^{-1}\cdot a_{\Delta}(F_{2, \chi})
\\& = 2\chi(2)\alpha^{-1}\cdot a_n(F_{\chi}) + \chi(2)^{-1}\cdot\frac{(\alpha-1)(\alpha+2)(\alpha^2-2)}{\alpha(\alpha^2+2\alpha-2)}\cdot  a_{\Delta}(F_{2, \chi})
\\&= 2\chi(2)\alpha^{-1}\cdot a_n(F_{\chi}) +  \chi(2)^{-1}\cdot\frac{(\alpha-1)(\alpha+2)(\alpha^2-2)}{\alpha(\alpha^2+2\alpha-2)}\cdot \frac{(\alpha^5 - 6\alpha^3 + 4\alpha^2 + 4\alpha - 4)}{(\alpha-1)(\alpha^4 - 6\alpha^2 + 4)} \cdot  a_{\Delta}(F_{\chi}) \text{.}
\end{align*}
Thus, we get

$$
(1-\chi(2)^{-1}\cdot\frac{(\alpha+2)(\alpha^2-2)(\alpha^5 - 6\alpha^3 + 4\alpha^2 + 4\alpha - 4)}{\alpha(\alpha^2+2\alpha-2)(\alpha^4 - 6\alpha^2 + 4)})\cdot a_{\Delta}(F_{\chi}) = 2\chi(2)\alpha^{-1}\cdot a_n(F_{\chi}) \text{.}
$$

This concludes the proof of \ref{prop_main_wt_3/2_iii}, and of the theorem.
\end{proof}

The following corollary of Theorem \ref{prop_main_wt_3/2} will be the key input for our proof of Theorem \ref{main_thm_paper}. 
\begin{cor}\label{cor_main_wt_3/2}
Let $\chi : (\Z/\ell\Z)^{\times}\rightarrow \C^{\times}$, $\chi' : (\Z/4\ell^2\Z)^{\times} \rightarrow \C^{\times}$ and $R = \Z[\chi] = \Z[\chi']$ be as defined in the notation of this section. Let $Q$ be a positive squarefree integer with $\gcd(Q,2p\ell)=1$. 

There exists $$g_{\chi,Q} \in S_{\frac{3}{2}}(\Gamma_1(4\ell^2Q^2), \chi')_{R/pR}$$ 
such that, for all $n\geq 1$ with $\left(\dfrac{n}{Q}\right) \neq -1$, we have $a_n(g_Q)=0$, and furthermore the following holds.
\begin{enumerate}[label=(\roman*)]
\item\label{cor_main_wt_3/2_i} For all prime $r \nmid 2Q\ell$, we have $T_{r^2}(g_{\chi, Q}) = \chi(r)\cdot(r+1)\cdot g_{\chi, Q}$.
\item\label{cor_main_wt_3/2_ii} For all $\Delta \in \mathcal{D}_{\ell}$ with $\left(\dfrac{\Delta}{Q}\right) = -1$, we have 
$$a_{\Delta}(g_{\chi, Q}) = \pm\chi(\sqrt{\Delta} \text{ modulo }\mathfrak{l})\cdot h(\Delta)\cdot \log_{\l, p}(\epsilon_{\Delta})  
\text{ (modulo }p\text{)}$$
for any prime $\mathfrak{l}$ above $\ell$ in $K=\Q(\sqrt{\Delta})$ and some sign $\pm$ may depending \emph{a priori} on $\Delta$. (The right-hand side is independent of the choice of $\mathfrak{l}$.)
\item\label{cor_main_wt_3/2_iii} Let $n\geq 1$ be such that $\left(\dfrac{n}{Q}\right) = -1$, and denote by $\Delta$ the discriminant of the real quadratic field $\Q(\sqrt{n})$. If $a_n(g_{\chi, Q})\neq 0$, then $a_{\Delta}(g_{\chi, Q}) \neq 0$ and $\Delta \in \mathcal{D}_{\ell}$.
\end{enumerate}
\end{cor}
\begin{proof}
Let $F$ and $F_{\chi}$ be as in Theorem \ref{prop_main_wt_3/2}. Define 
$$h_{\chi, Q} := \frac{1}{2}\cdot (F_{\chi}-F_{\chi}\otimes \left(\dfrac{\cdot}{Q}\right) - B_QF_{\chi})\in S_{\frac{3}{2}}(\Gamma_1(4\ell^2Q^2), \chi')_{R\otimes_{\Z}\C} \text{,}$$
where
$$F_{\chi}\otimes \left(\dfrac{\cdot}{Q}\right) := \sum_{n\geq 1} a_n(F_{\chi})\cdot \left(\dfrac{n}{Q}\right)\cdot q^n \in S_{\frac{3}{2}}(\Gamma_1(4\ell^2Q^2), \chi')_{R\otimes_{\Z} \C}$$
is the twist of $F_{\chi}$ by the primitive Dirichlet character $\left(\dfrac{\cdot}{Q}\right)$ of level $Q$ and
$$B_QF_{\chi}:= \sum_{n\geq 1} a_{nQ}(F_{\chi})q^{nQ} \in S_{\frac{3}{2}}(\Gamma_1(4\ell^2Q^2), \chi')_{R\otimes_{\Z} \C}$$
(\cf \eg \cite[\S 3]{Bruinier} for the definition of the twist and $B_Q$ operators). 

Thus, we have $a_n(h_{\chi, Q}) = a_n(F_{\chi})$ if $\left(\dfrac{n}{Q}\right)=-1$ and $a_n(h_{\chi, Q})=0$ else. Furthermore, by (\ref{eq_T_q^2_f}) and (\ref{eq_def_T_q^2}), for all prime $r$ with $r\nmid 2Q\ell$, we have 
$$T_{r^2}(h_{\chi, Q})=\chi(r)\cdot a_r(F)\cdot h_{\chi, Q} \text{.}$$ 
By Theorem \ref{prop_main_wt_3/2} \ref{prop_main_wt_3/2_i}, for all $n\geq 1$ we have 
$$a_n(h_{\chi, Q}) \in I_F\cdot (R \otimes_{\Z} \mathcal{O}_F \otimes_{\mathcal{O}_F} \mathcal{O}_{\mathfrak{P}_F}) \text{.}$$ We then let
$$g_{\chi, Q} := \sum_{n\geq 1} \widetilde{a_n(h_{\chi, Q})}\cdot q^n \in S_{\frac{3}{2}}(\Gamma_1(4\ell^2Q^2), \chi')_{R/pR} \text{,} $$
where $\widetilde{a_n(h_{\chi, Q})}$ denotes the image of $a_n(h_{\chi, Q})$ in $R/pR$ via the map (\ref{eq_I_F/I_F^2}). Point \ref{cor_main_wt_3/2_i} of the corollary then follows from the fact that $a_r(F) \equiv r+1 \text{ (modulo }I_F \text{)}$, while points \ref{cor_main_wt_3/2_ii} and point \ref{cor_main_wt_3/2_iii} follows from Theorem \ref{prop_main_wt_3/2} \ref{prop_main_wt_3/2_ii} and \ref{prop_main_wt_3/2_iii} respectively.
\end{proof}

\section{Proof of the main theorem}
We keep the notation of section \ref{section_half_int_wt}.

The goal of this section is to prove Theorem \ref{main_thm_paper}. Recall that, in the setting of Theorem \ref{main_thm_paper}, there is a squarefree integer $Q$ such that $\gcd(Q, 2p\ell)=1$, $Q \not\equiv \pm 1 \text{ (modulo }p\text{)}$ and there exists $\Delta_0 \in  \mathcal{D}_{\ell}$ such that $\left(\dfrac{\Delta_0}{Q}\right)=-1$ and
$h(\Delta_0)\cdot \log_{\l, p}(\epsilon_{\Delta_0}) \neq 0$,
for any prime $\mathfrak{l}\mid \ell$ in $\Q(\sqrt{\Delta_0})$.

Let 
$$g_{\chi, Q}\in S_{\frac{3}{2}}(\Gamma_1(4\ell^2Q^2), \chi)_{R/pR}$$ as in Corollary \ref{cor_main_wt_3/2}. 
By Corollary \ref{cor_main_wt_3/2} \ref{cor_main_wt_3/2_ii}, we conclude that $a_{\Delta_0}(g_{\chi, Q}) \neq 0$ (in $R/pR$). We let
$$\supp(g_{\chi, Q}) := \{n\geq 1 \text{ such that } a_n(g_{\chi, Q})\neq 0 \text{ (in }R/pR \text{)}\} \text{,}$$
which is non-empty.

The following result relies crucially on a result of Bruinier--Ono \cite[Theorem 3.1]{Bruinier_Ono} (\cf also \cite[Theorem 1]{Bruinier}).

\begin{lem}\label{lem_no_finite_support}
There does not exist finitely many squarefree integers $n_1$, ..., $n_t$ (for some $t\geq 1$) such that
$$\supp(g_{\chi, Q}) \subset \bigcup_{i=1}^t\{n_im^2, m \in \mathbf{N}\} \text{.}$$
Consequently, for any $C>1$ there exists an integer $n_0 \in \supp(g_{\chi, Q})$ such that $n_0=n_0'm^2$ for some $m\in \mathbf{N}$ and $n_0'$ squarefree and divisible by a prime $>C$.
\end{lem}
\begin{proof}
For the sake of a contradiction, assume that there are such (pairwise distincts) squarefree integers $n_1$, ..., $n_t$. Since $\supp(g_Q)$ consists of integers $n$ satisfying $\left(\dfrac{n}{Q}\right)=-1$, we can and do assume that for all $i \in \{1, ..., t\}$, we have $\left(\dfrac{n_i}{Q}\right)=-1$.
\begin{claim}
There exists a prime $r$ such that $r\nmid 2Qp\ell$,
$\left(\dfrac{n_i}{r}\right)=-1$ for all $i \in \{1, ..., t\}$ and $r \not\equiv \pm 1 \text{ (modulo }p\text{)}$. 
\end{claim}
Proof of the claim. By the quadratic reciprocity law, the condition $\left(\dfrac{n_i}{r}\right)=-1$ is equivalent to congruences for $r$ modulo $4n_i$. Together with the condition $r \not\equiv \pm 1 \text{ (modulo }p\text{)}$, we need to solve a system of congruences in $r$. This system has a solution given by $Q$, since we have $Q\not\equiv \pm 1 \text{ (modulo }p\text{)}$ and $\left(\dfrac{n_i}{Q}\right)=-1$ for all $i$. By Dirichlet's theorem on primes in arithmetic progressions, the claim follows.

Let $r$ be as in the claim. It follows from a slight generalization of \cite[Theorem 3.1]{Bruinier_Ono} that we have
$$(r-1)\cdot T_{r^2}(g_{\chi, Q}) = -\chi(r)\cdot (r+1)\cdot (r-1) \cdot g_{\chi, Q}\text{.}$$
Indeed, \cite[Theorem 3.1]{Bruinier_Ono} applies, as stated, to forms in $S_{\frac{3}{2}}(\Gamma_1(4\ell^2Q^2), \chi')_{\Z/p\Z}$ where $\chi'$ is a \emph{quadratic} character. However, the proof of that theorem purely relies on commutation relations between certain operators (especially the Atkin--Lehner involution), and can be applied identically to a form with coefficients in $R/pR$ instead of $\Z/p\Z$, and for a general (non necessarily quadratic) Nebentype $\chi'$. The only necessary change is to replace $\chi'$ with $\overline{\chi'}$ at certain places, but it is straightforward to check that the end result, \ie the identity for $(r-1)T_{r^2}f$, still holds (here, $p$ and $M$ in \cite[Theorem 3.1]{Bruinier_Ono} should be replaced by our $r$ and $p$ respectively).

Since $r \not\equiv \pm 1 \text{ (modulo }p\text{)}$ and $T_{r^2}(g_{\chi, Q}) = \chi(r)\cdot (r+1)\cdot g_{\chi, Q}$ (by Corollary \ref{cor_main_wt_3/2} \ref{cor_main_wt_3/2_i}), we get $g_{\chi, Q}=0$, which is not possible since, as explained above, our running assumption implies $\supp(g_{\chi, Q}) \neq \emptyset$.
\end{proof}

We let 
$$C=2\ell(\ell+1)Q\prod_{q\mid Q \atop q \text{ prime}}(q+1) \text{.}$$
By Lemma \ref{lem_no_finite_support}, there exists $n_0\geq 1$ such that $a_{n_0}(g_Q) \neq 0$ and whose squarefree part $n_0'$ is divisible by a prime factor $>C$. In what follows, we fix such a choice of $n_0$.
Let $S$ be the set of primes $r>C$ with $\gcd(r, Q\ell)=1$, such that for all $n\leq C$ we have $\left(\dfrac{n}{r}\right)=1$, and such that $\left(\dfrac{n_0}{r}\right)=-1$. It follows from Dirichlet's theorem and the existence of $n_0'$ as above that $S$ has positive density.

The following is the key lemma of our proof. Its proof takes inspiration in the technique of \cite{Ono} and \cite{Byeon}.

\begin{lem}\label{key_lem_main_thm}
For every $r \in S$, there exists $n\leq Cr$ coprime with $r$ such that $a_{nr}(g_{\chi, Q})\neq 0$.
\end{lem}
\begin{proof}
Assume, for the sake of a contradiction, that for all $n\leq Cr$ coprime to $r$, we have $a_{nr}(g_{\chi, Q})=0$. 

Let 
$$g_1 = \sum_{n\geq 1} a_{nr}(g_{\chi, Q})q^n$$
and
$$g_2 = \sum_{n\geq 1} a_n(g_{\chi, Q})q^{nr}.$$
By \cite[Propositions 1.3 and 1.5]{Shimura_annals}, both $g_1$ and $g_2$ are in $S_{3/2}(\Gamma_1(4\ell^2Q^2r),\chi'\cdot \left(\dfrac{r}{\cdot}\right))_{R/pR}$.

Our assumption is equivalent to $a_n(g_1)=0$ for all $n\leq C r$ coprime to $r$, and by construction, we have $a_n(g_2)=0$.  In particular, for $n\leq C r$ coprime to $r$, we have:
\begin{equation}\label{eq_coeff_n_nell^2_0}
a_{n}(g_1) = r\cdot \chi(r)\cdot a_{n}(g_2) \text{.}
\end{equation}

Recall (\cf \eg \cite[\S 3]{Bruinier}) that 
$$a_n(T_{r^2}g_{\chi, Q}) = a_{nr^2}(g)+\chi(r)\left(\dfrac{n}{r}\right)a_n(g_{\chi, Q})+r\cdot\chi(r^2)\cdot a_{\frac{n}{r^2}}(g_{\chi, Q}) \text{.}$$
By Corollary \ref{cor_main_wt_3/2} \ref{cor_main_wt_3/2_i}, we get, for all $n\geq 1$:
\begin{equation}\label{eq_coeff_n_nell^2_1}
a_{nr^2}(g_{\chi, Q}) = (\chi(r)\cdot(r+1)-\chi(r)\left(\dfrac{n}{r}\right))\cdot a_n(g_{\chi, Q}) - r\cdot \chi(r^2)\cdot a_{\frac{n}{r^2}}(g_{\chi, Q}).
\end{equation}
In particular, if $r^2 \nmid n$, we get
\begin{equation}\label{eq_coeff_n_nell^2_2}
a_{nr^2}(g_{\chi, Q}) = (\chi(r)\cdot(r+1)-\chi(r)\left(\dfrac{n}{r}\right))\cdot a_n(g_{\chi, Q}).
\end{equation}
or equivalently
$$
a_{nr}(g_1) = (\chi(r)\cdot(r+1)-\chi(r)\left(\dfrac{n}{r}\right))\cdot a_{nr}(g_2) \text{.}
$$
In particular, using the assumption $r\in S$, we get, for $n\leq C$:
\begin{equation}\label{eq_coeff_n_nell^2_3}
a_{nr}(g_1) = r\cdot \chi(r)\cdot a_{nr}(g_2) \text{.}
\end{equation}

We shall make use of the following consequence of the Sturm bound \cite{Sturm}: let $h\in S_{3/2}(\Gamma_1(4\ell^2Q^2r), \alpha)_{R/pR}$ with some Nebentypus $\alpha$. Assume that for all $n\leq 1+\frac{\frac{3}{2}+\frac{1}{2}}{12}[\SL_2(\Z):\Gamma_0(4\ell^2Q^2r)] = 1+\ell(\ell+1)(r+1)\prod_{q \mid Q \atop q \text{ prime}}q(q+1)$ we have $a_n(h)=0$, then $h = 0$. Sturm's bound (for forms with Nebentypus) is usually stated for integral weight modular forms, and modulo a prime ideal of the ring of integers of a number field (here $R$). Taking the multiplication with the usual weight $\frac{1}{2}$ and level $\Gamma_0(4)$ theta series, we reduce to the integral weight case. By decomposing $pR$ into product of powers of prime ideals, we reduce to the case of a prime power, which is proved in \cite[Corollary 2.15]{Sturm_bound_general}. Let us also note that
$$1+\ell(\ell+1)(r+1)\prod_{q \mid Q \atop q \text{ prime}}q(q+1)\leq Cr \text{.}$$

For $r > C$, there is no $n\leq C\cdot r$ divisible by $r$ with $r \mid \frac{n}{r}$. Thus,  (\ref{eq_coeff_n_nell^2_0}), (\ref{eq_coeff_n_nell^2_3}), the Sturm bound recalled above and the fact that $r \in S$ imply the following statement: 
$$g_1 = r\cdot \chi(r)\cdot g_2 \text{.}$$ 
In particular, this implies that for all $n\geq 1$ with $r\nmid n$, we have
\begin{equation}\label{eq_a_nell^3}
a_{nr^3}(g_1) = r\cdot \chi(r)\cdot a_{n r^3}(g_2) \text{.}
\end{equation}
Let us compute separately $a_{n_0 r^3}(g_1)$ and $a_{n_0 r^3}(g_2)$ to derive a contradiction. Using (\ref{eq_coeff_n_nell^2_1}) and (\ref{eq_coeff_n_nell^2_2}) successively, as well as the assumption $\left(\dfrac{n_0}{r}\right)=-1$, we get 
\begin{align*}
a_{n_0 r^3}(g_1) &= a_{n_0r^4}(g_{\chi, Q}) 
\\&= (\chi(r)\cdot(r+1)-\chi(r)\left(\dfrac{n_0r^2}{r}\right))\cdot a_{n_0r^2}(g_{\chi, Q}) - r\cdot \chi(r^2)\cdot a_{n_0}(g_{\chi, Q})
\\&= (\chi(r)^2\cdot (r+1)\cdot (r+1-\left(\dfrac{n_0}{r}\right))-r\cdot\chi(r^2))\cdot a_{n_0}(g_{\chi, Q})
\\& = \chi(r^2)\cdot ((r+1)\cdot(r+2)-r)\cdot a_{n_0}(g_{\chi, Q}) \text{.}
\end{align*}
On the other-hand, we have
\begin{align*}
a_{n_0 r^3}(g_2) &= a_{n_0r^2}(g_{\chi, Q}) 
\\& = \chi(r)\cdot (r+1-\left(\dfrac{n_0}{r}\right))\cdot a_{n_0}(g_{\chi, Q})
\\& = \chi(r)\cdot (r+2)\cdot a_{n_0}(g_{\chi, Q}) \text{.}
\end{align*}
Together with (\ref{eq_a_nell^3}), this yields:
$$\chi(r^2)\cdot ((r+1)\cdot(r+2)-r)\cdot a_{n_0}(g_{\chi, Q}) = \chi(r^2)\cdot r\cdot (r+2)\cdot a_{n_0}(g_{\chi, Q}) $$
\ie
$$2\chi(r^2)\cdot a_{n_0}(g_{\chi, Q}) = 0 \text{.}$$
Since $2\chi(r^2)$ is invertible in $R/pR$ (as $p>2$), we get $a_{n_0}(g_{\chi, Q}) = 0$, which is a contradiction since by construction of $n_0$, we have $a_{n_0}(g_{\chi, Q})\neq 0$.
\end{proof}

We can now finish the proof of Theorem \ref{main_thm_paper}.

Take $r$ in $S$. By Lemma \ref{key_lem_main_thm}, there exists $n_{r} \leq C r$, coprime with $r$, such that $a_{n_{r}r}(g_{\chi, Q})\neq 0$.
By Corollary \ref{cor_main_wt_3/2} and Proposition \ref{prop_link_h_ell_unit}, if $\Delta_r$ denotes the discriminant of $\Q(\sqrt{n_{r}r})$, it follows that $p\nmid h_\ell^-(\Delta_r)$ and $\displaystyle{\left(\frac{\Delta_r}{Q}\right)=-1}$.
Observe now that $\Delta_r \leq n_{r}r\leq C r^2$. Consequently, as $X \rightarrow \infty$, there are at least $\gg \displaystyle{\frac{\sqrt{X}}{\log(X)}}$ discriminants $\Delta_r$, possibly with multiplicity, due to the prime number theorem and the fact that $S$ has positive density. We still need to verify that there are not too many pairs $r\neq r'$ such that $\Delta_{r} = \Delta_{r'}$.

Take now $r_i<r_{j} < r_{k} $ three different primes in $S$.
Hence $r_{j}$ and $r_{k}$ are  greater than $C$, and they cannot divide both $n_{i}$; indeed, otherwise $C r_k < n_i \leq C r_i$, which leads to a contradiction.
Hence, among the discriminants $\Delta_i, \Delta_j, \Delta_k$, at least two are different. Consequently, the number of discriminants $0\leq \Delta \leq X$ for which $p \nmid h_\ell^-(\Delta)$ and  $\displaystyle{\left(\frac{\Delta}{Q}\right)=-1}$   is $\gg \displaystyle{\frac{\sqrt{X}}{\log(X)}}$ when $X \rightarrow \infty$.

\section{Proof of Corollary \ref{cor_BSD}}\label{section_cor_BSD}
Let us restrict to the situation where $\ell=11$ and $p=5$.
Let $E=X_0(11)$, which is an elliptic curve over~$\Q$. By \cite[Corollary 1.8]{Lecouturier_Wang}, if $\Delta>0$ is a fundamental discriminant such that $5\nmid \Delta$ and $\left(\dfrac{\Delta}{11}\right)=1 $, then $5\nmid h_{11}^-(\Delta)$ if and only if $\Sel_5(E^{(\Delta)}/\Q)=0$, where $E^{(\Delta)}$ is the quadratic twist of $E$ by $\Delta$ and $\Sel_5(E^{(\Delta)}/\Q)$ is the $5$-Selmer group of $E^{(\Delta)}$. Furthermore, in this case, the $5$-part of BSD holds.

Therefore, Corollary \ref{cor_BSD} would follow from Corollary \ref{main_cor} if we did not have the restriction $\left(\dfrac{\Delta}{5}\right)=-1$. The characters $\chi$ and $\chi'$ can be taken quadratic, as $11\equiv 3 \text{ (modulo }4\text{)}$. Furthermore, the form $F_{\chi}$ of Theorem \ref{prop_main_wt_3/2} has coefficients in $\Z$, as it is given explicitly in terms of generalized theta series (\cf \cite[\S 10.7]{Mao} and \cite[\S 4.1.1]{Villegas}), and we know that $F_{\chi} \equiv \theta_{\chi} \text{ (modulo }5\text{)}$. We then let 
$$
\tilde{g}_{\chi} := \frac{F_{\chi}-\theta_{\chi}}{5} \in S_{\frac{3}{2}}(\Gamma_1(4\cdot 11^2), \chi')
$$
and
$$
\tilde{g}_{\chi, 5} := \sum_{n \geq 1 \text{ s.t. }\left(\dfrac{n}{5}\right)=-1} a_n(\tilde{g}_{\chi})q^n \in S_{\frac{3}{2}}(\Gamma_1(4\cdot 11^2\cdot 5^2), \chi') \text{,}
$$
which have Fourier coefficients in  $\Z$. Note that we can allow the level to be divisible by $5$, as the coefficient field is $\C$, in which $5$ is invertible. 

By Theorem \ref{prop_main_wt_3/2}, for a fundamental discriminant $\Delta \in \mathcal{D}_{11}$ with $\left(\dfrac{\Delta}{5}\right)=-1$, we have $5 \mid a_{\Delta}(\tilde{g}_{\chi, 5})$ if and only if $5 \nmid h_{11}^-(\Delta)$. Then an identical argument as in Lemma \ref{key_lem_main_thm} using the Sturm bound shows that the analogue of Theorem \ref{main_thm_paper} with the additional condition $\left(\dfrac{\Delta}{5}\right)=-1$ holds, under the condition that there exists a fundamental discriminant $\Delta_0 \in \mathcal{D}_{11}$ satisfying $\left(\dfrac{\Delta_0}{5}\right)=-1$ and which is divisible by a prime divisor $>C = 2\cdot 11\cdot (11+1)\cdot 5\cdot (5+1) = 7920$. 
Such a $\Delta_0$ indeed exists, as it is straightforward to check that $\Delta_0 = 8237$ (which is prime) works.
\bibliography{biblio}
\bibliographystyle{plain}

\end{document}